\documentclass[12pt]{amsart}

\usepackage[a4paper,margin=23mm,top=30mm]{geometry}

\usepackage{amsfonts, amsmath, amssymb, amsgen, amsthm, amscd}
\usepackage{newtxtext,newtxmath}
\usepackage{mathtools}
\usepackage[utf8]{inputenc} 

\usepackage{color}

\usepackage[all]{xy}

\usepackage{hyperref}
\usepackage{mathtools,slashed}

\usepackage{enumerate}

\def\a{\alpha}
\def\b{\beta}
\def\d{\delta}
\def\D{\Delta}
\def\g{\gamma}

\def\s{\sigma}

\def\t{\theta}

\def\ve{\varepsilon}
\def\vp{\varphi}

\def\Id{\mathop{\rm Id}\nolimits}

\def\ot{\otimes}

\def\rt{\triangleright}
\def\lt{\triangleleft}

\def\lbiprod{{>\!\!\!\triangleleft\kern-.33em\cdot\, }}
\def\rbiprod{{\cdot\kern-.33em\triangleright\!\!\!<}}

\def\D{\Delta}

\def\Id{\mathop{\rm Id}\nolimits}

\newcommand{\ns}[1]{~\hspace{-4pt}_{_{{<#1>}}}}

\usepackage{enumerate}

\newcommand{\C}[1]{\mathcal{#1}}
\newcommand{\B}[1]{\mathbb{#1}}

\newcommand{\ie}{{\it i.e.\/}\ }

\numberwithin{equation}{section}

\newtheorem{theorem}{Theorem}[section]
\newtheorem{proposition}[theorem]{Proposition}
\newtheorem{lemma}[theorem]{Lemma}

\theoremstyle{definition}

\newtheorem{remark}[theorem]{Remark}
\newtheorem{example}[theorem]{Example}
\newtheorem{definition}[theorem]{Definition}

\title{Matched pairs of $m$-invertible Hopf quasigroups}

\author{M. Hassanzadeh}
\address{University of Windsor, ON, Canada}
\email{mhassan@uwindsor.ca}

\author{S. Sütlü}
\address{Işık University, Istanbul, Turkey}
\email{serkan.sutlu@isikun.edu.tr}

\begin{document}

\begin{abstract}
The matched pair theory (of groups) is studied for a class of quasigroups; namely, the $m$-inverse property loops. The theory is upgraded to the Hopf level, and the ``$m$-invertible Hopf quasigroups'' are introduced.
\end{abstract}

\maketitle

\tableofcontents

\section*{Introduction}

One of the main motivations of the theory of quasigroups may be considered to be the extension of the representation theoretical properties of the groups on the level of quasigroups; such as the character theory \cite{John92,Smit90}, module theory \cite{Smit86}, or homogeneous spaces \cite{Smit99,Smit01}. See also \cite{GrisZava05,GagoHall05,GrisZava09}.

Not much later, it was discovered that there are a plethora of areas for quasigroups to apply. Among others, an incomplete list may consists of the coding theory (see, for instance, \cite{GonsKousMarkNech98} for the quasigroup-based MDS codes, and \cite{MullShch02,MullShch04} for the quasigroup point of view towards the codes with one check symbol, as well as \cite{Grie96}), cryptology \cite{DiffHell76,GrosSys10,Shch09}, and combinatorics \cite{Benn83,HeinKlin09,Keed94,KuhlSchr16}.

In order to shed further light on the well deserved analysis of the quasigroups, we shall develop in the present paper the matched pair construction for these non-associative structures. The matched pair theory was introduced, initially, for groups in order to recover the structure of a group in terms of two subgroups with mutual actions, \cite{Brin05,KosmSchwMagr88,LuWein90,Maji90-II,Szep51,Zapp42}. More precisely, given a pair of groups $(G,H)$ with mutual actions
\[
\rt:H\times G \to G, \qquad \lt:H\times G \to H
\]
satisfying
\begin{align*}
& y\rt (xx') = (y\rt x)((y\lt x)\rt x'), \qquad y\rt 1 = 1, \\
& (yy') \lt x = (y\lt (y'\rt x))(y'\lt x), \qquad 1\lt x = 1,
\end{align*}
for any $x,x'\in G$, and any $y,y'\in H$, the cartesian product $G\bowtie H:= G\times H$ is a group with the multiplication
\[
(x,y)(x',y') = (x(y\rt x'), (y\lt x)y'),
\]
and the unit $(1,1)\in G\times H$. In this case, the pair $(G,H)$ of groups is called a matched pair of groups.

As for the quasigroups, there are many such constructions. To begin with, there are of course the direct product construction \cite{Cohn-book,Smith-book,NemeKepk71,Jeze75,Bely87,Bely90}, and the semi-direct product construction \cite{Shcherbacov-book, Sabi91,Burd76}. There is also the crossed product construction \cite{Bruc44,Bruc44-II,Belousov-book}, which is referred as quasi-direct product in \cite{Wils75}. Considering these as the binary crossed products, there are, on top of these, the $n$-ary crossed products \cite{Bori79}. The other generalizations goes under the titles of the generalized crossed product \cite{BelyLump85}, and the generalized singular direct product \cite{Sade53,Sade60}. Finally, there is the Sabinin's product \cite{Sabi72,Sabinin-book} and its generalization \cite{Burd76,Sabi96}. We refer the reader also to \cite{CheinPflugfelderSmith-book}.

The matched pair construction (for $m$-inverse property loops) that we shall develop here, on the other hand, is one based on ``mutual  actions'' of two such objects. We shall, furthermore, be able to relate this construction to the matched pair of groups; which will enable us to produce an ample amount of examples motivated from the matched pairs of groups.

Let us note that the matched pair theory of groups suit also to Hopf algebras, the quantum analogues of  groups, \cite{Maji90,Majid-book,Take81}. Just as well, there will be a Hopf analogue of the theory we shall develop here.

In \cite{KlimMaji10}, see also \cite{KlimMaji10-II}, the authors managed to develop successfully a not-necessarily associative (but coassociative, counital, and unital) (co)algebra $H$, that they call a Hopf quasigroup, with a map $S:H\to H$ satisfying compatibility conditions more general then those satisfied by the antipode of a Hopf algebra. It is further shown that $kQ$ is a Hopf quasigroup if $Q$ is an inverse property (IP) loop, and that for any Hopf quasigrop $H$, the set $G(H)$ of group-like elements form an IP loop.

Considering the Hopf algebras as linearizations of groups, one sees the antipode of a Hopf algebra as the manifestation of the inversion on a group. This point of view is precisely what has been studied in \cite{KlimMaji10-II,KlimMaji10}, where the authors succesfully developed the correct axioms for the ``antipode'' of the quantum analogue of an IP loop. Looking from a similar perspective, in the present paper we shall develop the quantum analogue of a strictly larger family; the $m$-inverse property loops. We note from \cite{KeedShch03} that the $m$-inverse property is general enough to encompass the weak-inverse-property (WIP) of \cite{Baer39}, as well as the crossed-inverse (CI) property of \cite{Artz59}.

The paper is organized as follows.

Section \ref{sect-QIP} below is about the inverse properties on quasigroups, and serves to fix the basic definitions of the main objects of study. To this end, in Subsection \ref{subsect-quasigr} we collect the definitions of quasigroups and loops, while in Subsection \ref{subsect-IPQ} we recall briefly the various inverse properties on quasigroups (with a special emphasis on the $m$-inverse property). 

Section \ref{sect-matched-pair-quasigr} is where we develop the matched pair theory for the $m$-inverse property loops. Based on the lack of literature on semi-direct product of quasigroups (in the sense that one quasigroup acts on the other, see for instance Proposition \ref{prop-semi-direct-quasi} and Proposition \ref{prop-semi-direct-quasi-II} below), and for the convenience of the reader, we begin with a recollection of the basic results (Theorem \ref{thm-quasigr-direct-prod} and Theorem \ref{thm-quasigr-direct-prod-II}) on the direct products of quasigroups in Subsection \ref{subsect-direct-prod-quasigr}, and then extend it to the semi-direct products of $m$-inverse property loops (Proposition \ref{prop-semi-direct-quasi-III} and Proposition \ref{prop-semi-direct-quasi-IV}). Finally, we achieve the full generality (proving our main results on the quasigroup level) in Subsection \ref{subsect-matched-pair-quasigr}, and succeed the matched pair construction for the $m$-inverse property loops (Proposition \ref{prop-matched-pair-quasi} and Proposition \ref{prop-matched-pair-loops}). We also discuss the universal property of this construction in Proposition \ref{prop-matched-pair-quasi-universal}, as an analogue of \cite[Prop. 6.2.15]{Majid-book} for the $m$-inverse property loops.

Section \ref{sect-linearization}, finally, is reserved for the quantum counterparts of the main results of Subsection  \ref{subsect-matched-pair-quasigr}. Accordingly, in Subsection \ref{subsect-m-invtb-Hopf-quasigr} we introduce the notion of $m$-invertible Hopf quasigroup in Definition \ref{def-m-inv-loop}. Then, in Subsection \ref{subsect-matched-pair-Hopf-quasi} we establish the matched pair theory for the $m$-invertible Hopf quasigroups (Proposition \ref{prop-matched-pair-Hopf-quasigrs}), along with a suitable version (Proposition \ref{prop-m-inv-loop-univ}) of \cite[Thm. 7.2.3]{Majid-book}.

\subsubsection*{Notation and Conventions} We shall adopt the Sweedler's notation (suppressing the summation) to denote a comultiplication; $\D:A\to A\ot A$, $\D(a):=a\ns{1}\ot a\ns{2}$. For the sake of simplicity, we shall also denote, occasionally, an element in the cartesian product $A\times B$, or tensor product $A\ot B$ as $(a,b)$, rather than $a\ot b$.

\section{Quasigroups with inverse properties}\label{sect-QIP}

In this section we shall discuss the semi-direct product, and then the matched pair constructions on two large classes of semigroups; namely the $m$-inverse loops, and the Hom-groups. To this end, we review the basics of the quasigroup theory first. We shall then focus on the 
inverse-properties (IP) over quasigroups, in order to be able to recall the $(r, s, t)$-inverse quasigroups, as well as the $m$-inverse loops. Finally, on the other extreme, we shall recall/review the basics of the Hom-groups.

\subsection{Quasigroups}\label{subsect-quasigr}~

A ``quasigroup'' is a set $Q$  with a multiplication such that  for all $a , b\in Q$,  there exist unique elements $x , y\in Q$ such that
\[
a x = b, \qquad \qquad y a = b.
\]
In this case, $x = a\backslash b$ is called the left division, and $y= b\slash a$ the right division.

Given two quasigroups $Q$ and $Q'$, a ``quasigroup homotopy'' from $Q$ to $Q'$ is a triple $(\a, \b, \g)$ of maps $Q \to Q'$ such that $\a(x) \b(y)= \g(xy)$ for all $x, y\in Q$. In case $\a = \b = \g$, then we arrive at the notion of a ``quasigroup homomorphism''. On the other hand, a ``quasigroup isotopy'' is a quasigroup homotopy $(\a,\b,\g)$ such that all three maps are bijective.

%

A quasigroup $Q$ with a distinguished idempotent element $\d \in Q$ is called a ``pointed idempotent quasigroup'', or in short, a ``pique'', \cite{CheinPflugfelderSmith-book}. A pique $(Q,\d)$ is called a ``loop'' if the idempotent element $\d \in Q$ acts like an identity, \ie
\[
x\d = \d x = x
\]
for any $x \in Q$. It, then, follows that the idempotent element $\d \in Q$ is unique, and that any $x \in Q$ has a unique left inverse
\[
x^\lambda := \d \slash x, \qquad x^\lambda x = \d
\]
as well as a unique right inverse
\[
x^\s := x \backslash \d, \qquad xx^\s = \d.
\]
A loop $Q$ is said to have two-sided inverses if $x^\lambda = x^\s$ for all $x \in Q$. Furthermore, a loop $Q$ is said to have the ``left inverse property'' if
\[
x^\lambda(xy) = y, \qquad \forall \, x,y \in Q,
\]
and similarly $Q$ is said to have the ``right inverse property'' if
\[
(yx)x^\s = y, \qquad \forall \, x,y \in Q.
\]
Finally, a loop is said to have the ``inverse property'' if it has both the left inverse property and the right inverse property. Such loops are also called the ``IP loops''.

Given a pique $(Q,\d)$, there corresponds a loop $B(Q)$ - called the ``corresponding loop'' or ``cloop'' - with the multiplication
\[
x \ast y := (x\slash \d)(\d\backslash y),
\]
for any $x,y \in Q$, and the identity element $\d \in Q$. We note that it is possible to recover the multiplication on a pique from the one on the cloop as
\[
xy := (x\d)\ast (\d y),
\]
see, for instance, \cite{RomaSmit06}. 

Finally, a pique $(Q,\d)$ is called ``central'' if $B(Q)$ is an abelian group, and the set of all left and right multiplications of $Q$ that fix the idempotent element $\d \in Q$ is the group ${\rm Aut}\big(B(Q)\big)$.

A convenient way to construct quasigroups, out of groups, is the cocycle-type group extensions, \cite{Belo65}, see also \cite[Subsect. 1.6.2]{Shcherbacov-book}. 

\begin{example}\label{ex-quasigr-construction}
Let $G$ be a group, $(V,+)$ an abelian group with a right action $\lt:V\times G \to V$, $(v,x)\mapsto v\lt x$. Then, given any $\vp:G\times G\to V$, the operation
\begin{equation}\label{multp-quasigr}
(x,v)(x',v') := (xx', \vp(x,x') + v\lt x' + v')
\end{equation}
is associative on $G \ltimes_\vp V: = G\times V$ if and only if 
\begin{equation}\label{cocycle-cond}
d\vp (x,x',x'') := \vp(x',x'') - \vp(xx',x'') + \vp(x,x'x'') - \vp(x,x')\lt x'' = 0,
\end{equation}
that is, $\vp:G\times G \to V$ is 2-cocycle in the group cohomology of $G$, with coefficients in $V$; in other words, $\vp\in H^2(G,V)$. As such, giving up the cocycle condition \eqref{cocycle-cond} we arrive at the quasigroup $G \ltimes_\vp V$ with the multiplication \eqref{multp-quasigr}.
\end{example}

Similarly, we may construct a loop.

\begin{example}\label{ex-loop}
Considering the quasigroup $G \ltimes_\vp V$ of Example \ref{ex-quasigr-construction}, we see at once that $(1,0) \in G \ltimes_\vp V$ acts as unit, with respect to \eqref{multp-quasigr}, if and only if
\begin{equation}\label{quasi-0}
\vp(1,x) = 0 = \vp(x,1)
\end{equation}
for any $x\in G$. Hence, given a group $G$, an abelian group $(V,+)$ with a right action $\lt:V\times G \to V$, and a mapping $\vp:G\times G \to V$ satisfying \eqref{quasi-0} is a loop.
\end{example}

We shall, for the sake of simplicity, drop the right action (that is, we shall assume the right action to be trivial) on the sequel, and consider the examples of the form $G\times_\vp V$, with the multiplication
\begin{equation}\label{multp-quasigr-II}
(x,v)(x',v') := (xx', \vp(x,x') + v + v').
\end{equation}

\subsection{Inverse properties on quasigroups}\label{subsect-IPQ}~

In the present subsection we shall recall the inverse properties on quasigroups, and in particular, on loops.

Along the lines of \cite{KeedShch03}, see also \cite{Baer39}, a loop $Q$ is said to have the ``weak-inverse property'' (WIP) if there is a permutation $J:Q \to Q$ such that 
\begin{equation}\label{perm-J}
xJ(x) = \d,
\end{equation}
and that
\begin{equation}\label{weak-inv-prop}
xJ(yx) = J(y),
\end{equation}
for any $x,y \in Q$. Dropping the condition \eqref{perm-J}, a quasigroup with \ref{weak-inv-prop} is called a WIP quasigroup. 

Similarly, a loop / quasigroup $Q$ is said to have the ``crossed-inverse property'' (CI property) if \eqref{weak-inv-prop} is replaced by
\begin{equation}\label{crossed-inv-prop}
(xy)J(x) = y.
\end{equation}
We refer the reader to \cite{Keed99} for the applications of the CI quasigroups in cryptography. 

On the other hand, the loop / quasigroup $Q$ has the ``$m$-inverse property'' if \eqref{weak-inv-prop}, or \eqref{crossed-inv-prop}, is now substituted with 
\begin{equation}\label{m-inv}
J^m(xy)J^{m+1}(x) = J^m(y),
\end{equation}
where $m\in \B{Z}$, \cite{KarkKark76}. 

Finally, we recall that the loop / quasigroup $Q$ is said to have the ``$(r,s,t)$-inverse property'' if \eqref{weak-inv-prop}, \eqref{crossed-inv-prop}, or \eqref{m-inv}, is exchanged with 
\begin{equation}\label{rst-inv}
J^r(xy)J^s(x) = J^t(y),
\end{equation}
where $r,s,t \in \B{Z}$, \cite{KeedShch03}.

\begin{remark}
The condition \eqref{rst-inv} generalizes those given by \eqref{weak-inv-prop}, \eqref{crossed-inv-prop}, or \eqref{m-inv}. More precisely, the weak-inverse property is a $(-1,0,-1)$-inverse property, \cite{KeedShch03}, and a crossed-inverse property is nothing but a 0-inverse property; where, in general an $m$-inverse property is an $(m,m+1,m)$-inverse property. 

On the other hand, it is observed in \cite{KeedShch02} that every $(r,s,t)$-inverse loop is an $(r,r+1,r)$-inverse loop, that is, an $r$-inverse loop. Though, on the level of quasigroups, there are proper $(r,s,t)$-inverse quasigroups, \cite{KeedShch03}.
\end{remark}

\begin{remark}\label{rk-m+uh}
It is critical to recall from \cite[Rk. 2.2]{KeedShch03} that if $Q$ is an $(r,s,t)$-inverse quasigroup with the permutation $J:Q \to Q$ so that $J^h\in {\rm Aut}(Q)$, then $Q$ is an $(r+uh,s+uh,t+uh)$-inverse quasigroup for any $u\in \B{Z}$. 
\end{remark}

Let us finally discuss an odd-invertible loop.
\begin{example}\label{ex-odd-invtble-loop}
Let us consider the loop $G\times_\vp V$ of Example \ref{ex-loop}. Let also
\begin{equation}\label{odd-invtble-J}
J:G\times_\vp V \to G\times_\vp V, \qquad J(x,v) := (x^{-1},-v).
\end{equation}
It is quite clear then that $J^2 = \Id_{G\times V} \in {\rm Aut}(G\times_\vp V)$. Accordingly, we see at once that
\[
(x,v)J(x,v) = (1,0)
\]
if and only if
\begin{equation}\label{quasi-I}
\vp(x,x^{-1}) = 0
\end{equation}
for any $x\in G$, and that for any $m=2\ell+1$,
\[
J^m((x,v)(x',v'))J^{m+1}(x,v) = J^m(x',v')
\]
if and only if
\[
J((x,v)(x',v'))(x,v) = J(x',v'),
\]
if and only if
\begin{equation}\label{quasi-II}
\vp(x'^{-1}x^{-1}, x) = \vp(x,x')
\end{equation}
for any $x,x'\in G$. 

To sum up, we may say that given any group $G$, an abelian group $(V,+)$, and any $\vp:G\times G\to V$ satisfying \eqref{quasi-0}, \eqref{quasi-I}, and \eqref{quasi-II}, $G\times_\vp V$ is an $(2\ell+1)$-invertible loop with \eqref{odd-invtble-J} for any $\ell \in \B{Z}$.
\end{example}

\section{Matched pairs of $m$-invertible loops}\label{sect-matched-pair-quasigr}

In this subsection we shall introduce the matched pair theory for the quasigroups with the $m$-inverse property. The theory that we shall develop here will thus generalize the direct product theory in \cite[Sect. 5]{KeedShch03}, and the semi-direct product theory in \cite[Sect. 1.6.2]{Shcherbacov-book} for quasigroups. 

\subsection{Direct products of quasigroups}\label{subsect-direct-prod-quasigr}~

To this end we shall first recall the direct product theory from \cite[Sect. 5]{KeedShch03} . In the utmost generality, let $Q_1$ be an $(r_1,s_1,t_1)$-inverse quasigroup with the permutation $J_1:Q_1\to Q_1$ , and let $Q_2$ an $(r_2,s_2,t_2)$-inverse quasigroup with $J_2:Q_2\to Q_2$. Then the direct product $Q_1\times Q_2$ is defined to be the quasigroup with the permutation $J_1\times J_2:Q_1\times Q_2 \to Q_1 \times Q_2$, and the multiplication given by $(q_1,q_2)(q'_1,q'_2) := (q_1q'_1,q_2q'_2)$.

Along the lines of \cite[Sect. 5]{KeedShch03}, let $J_1^{h_1} \in {\rm Aut}(Q_1)$ and $J_2^{h_2} \in {\rm Aut}(Q_2)$. In the case that $Q_1$ is an $m_1$-inverse quasigroup and $Q_2$ is an $m_2$-inverse quasigroup, the structure of $Q_1\times Q_2$ is given in \cite[Thm. 5.1]{KeedShch03}, that we recall below.

\begin{theorem}\label{thm-quasigr-direct-prod}
Let $Q_1$ is an $m_1$-inverse quasigroup with the permutation $J_1:Q_1 \to Q_1$ so that $J_1^{h_1} \in {\rm Aut}(Q_1)$, and $Q_2$ is an $m_2$-inverse quasigroup with $J_2:Q_2 \to Q_2$ such that $J_2^{h_2} \in {\rm Aut}(Q_2)$. Then $Q_1 \times Q_2$ is an $m$-inverse quasigroup with $J_1\times J_2: Q_1 \times Q_2 \to Q_1 \times Q_2$, for any $m \in \B{Z}$ that satisfies 
\begin{align}\label{cong-eqn}
\begin{split}
& m \equiv m_1\,\, ({\rm mod}\,h_1), \\
& m \equiv m_2\,\, ({\rm mod}\,h_2).
\end{split}
\end{align}
\end{theorem}

As is noted in the proof of  \cite[Thm. 5.1]{KeedShch03}, a solution to \eqref{cong-eqn} exists if and only if there is $\ell \in \B{N}$ such that $m_1-m_2 = (h_1,h_2)\ell$. Here $(h_1,h_2)$ refers to the greatest common divisor of $h_1 \in \B{Z}$ and $h_2 \in \B{Z}$.

If, on the other hand, $Q_1$ is an $(r_1,s_1,t_1)$-inverse quasigroup, and $Q_2$ is an $(r_2,s_2,t_2)$-inverse quasigroup, the structure of the direct product is given by \cite[Thm. 5.2]{KeedShch03}, which we recall now.

\begin{theorem}\label{thm-quasigr-direct-prod-II}
Let $Q_1$ is an $(r_1,s_1,t_1)$-inverse quasigroup with the permutation $J_1:Q_1 \to Q_1$ so that $J_1^{h_1} \in {\rm Aut}(Q_1)$, and $Q_2$ is an $(r_2,s_2,t_2)$-inverse quasigroup with $J_2:Q_2 \to Q_2$ such that $J_2^{h_2} \in {\rm Aut}(Q_2)$. Then $Q_1 \times Q_2$ is an $(r,s,t)$-inverse quasigroup with $J_1\times J_2: Q_1 \times Q_2 \to Q_1 \times Q_2$, if there are $u_1,u_2\in \B{Z}$ such that
\[
r-r_1 = s-s_1 = t-t_1 = u_1h_1, \qquad r-r_2 = s-s_2 = t-t_2 = u_2h_2.
\]
\end{theorem}

\subsection{Semi-direct products of $m$-invertible loops}\label{subsect-semi-direct-prod-quasigr}~

As for the semi-direct products of quasigroups, there seems to be no approach involving the notion of an action of a quasigroup on another. A semi-direct product construction, using groups, is the one given in \cite{RobiRobi93,KargapolovMerzljakov-book}, see also \cite[Sect. 1.6.2]{Shcherbacov-book} which we recall below.

\begin{proposition}\label{prop-semi-direct-quasi}
Let $(G,+)$ and $(H,\cdot)$ be two groups, and $\t:G\to {\rm Aut}(H)$. Then, $G\times H$ is a quasigroup with the multiplication
\[
(g,h)(g',h') := (g+g', \t(g')(h)\cdot h').
\]
\end{proposition}

The construction given in \cite{Sabinin-book} uses a quasigroup, and its transassociant.

\begin{proposition}\label{prop-semi-direct-quasi-II}
Let $Q$ be a quasigroup, and $H$ be the group generated by $\{\ell(q,q')\mid q,q' \in Q\}$, where $\ell(q,q') := L^{-1}_{qq'} \circ L_q \circ L_{q'}$, and $L_q:Q \to Q$, $L_q(r) := qr$, is the left translation. Then, $Q\times H$ is a quasigroup with the multiplication given by
\[
(q,h)(q',h') := (qh(q'), \ell(q,h(q'))\circ m_{q'}(h)\circ h \circ h'),
\]
where, for any $q\in Q$ and any $h \in H$,
\[
m_q(h):=L^{-1}_{h(q)}\circ h\circ L_q\circ h^{-1}.
\]
\end{proposition}

Let us note also that this was the point of view considered in \cite{KlimMaji10-II,KlimMaji10}.

None of these constructions lead to a possible discussion on the matched pairs of quasigroups. We thus adopt the following (more general, given in terms of quasigroup homomorphisms) definition given in \cite[Def. 1.318]{Shcherbacov-book}.

\begin{definition}\label{def-semi-direct}
A quasigroup $Q$ is called the semi-direct product of two quasigroups $R$ and $S$, if there is a (quasigroup) homomorphism $h:Q \to S$, such that the kernel $\ker(h) =R$, and that $h|_S = \Id_S$. In this case, $Q $ is denoted by $R \rtimes S$.
\end{definition}

The motivating examples are the ones discussed within the following propositions below, on the level of ($m$-inverse) loops, and Hom-groups.

\begin{proposition}\label{prop-semi-direct-quasi-III}
Let $R$ and $S$ be two loops, and let $\vp:S \times R \to R $ be a map satisfying $\vp(s,\d) = \d$ and $\vp(\d,r) = r$. Then a loop $Q$ is isomorphic to the loop $R \rtimes S:=R \times S$ with the multiplication given by
\begin{equation}\label{semi-direct-multp}
(r,s)(r',s') := (r\vp(s,r'), ss'),
\end{equation}
if and only if there are quasigroup homomorphisms $i_S:S \to Q$, $i_R:R \to Q$, $p_S:Q \to S$, and a map $p_R:Q \to R$ satisfying the Moufang-type identities 
\begin{align}\label{p_R-Moufang-I}
\begin{split}
& p_R\big(\big(i_R(r)i_S(s)\big)\big(i_R(r')i_S(s')\big)\big) = p_R(i_R(r))\Big(\big(p_R(i_S(s)i_R(r'))\big)p_R(i_S(s'))\Big) = \\
&\Big(p_R(i_R(r))\big(p_R(i_S(s)i_R(r'))\big)\Big)p_R(i_S(s')),
\end{split}
\end{align}
as well as $p_R\circ i_R = \Id_R$ and $p_S\circ i_S = \Id_S$, such that $R\rtimes S \to Q$, $(r,s) \mapsto i_R(r)i_S(s)$ and $Q \to R\rtimes S$, $q\mapsto (p_R(q),p_S(q))$ are inverse to each other.
\end{proposition}

\begin{proof}
Letting $\Phi:Q\to R \rtimes S$ to be the (quasigroup) isomorphism, we consider the mappings  
\[
i_R:R \to Q, \quad i_R(r) := \Phi^{-1}(r,\d), \qquad i_S:S \to Q, \quad i_S(s) := \Phi^{-1}(\d,s)
\]
and
\[
p_R:Q \to R, \quad p_R(q) := \pi_1(\Phi(q)), \qquad p_S:Q \to S, \quad p_S(q) := \pi_2(\Phi(q)),
\]
where $\pi_i$'s denote the projection onto the $i$th component. It is evident that 
\[
(p_R \circ i_R)(r) = \pi_1(r,\d) = r,
\]
for any $r \in R$, as such $p_R \circ i_R = \Id_R$. Similarly, $p_S\circ i_S = \Id_S$. We further see that
\[
i_S(ss') = \Phi^{-1}(\d,ss') = \Phi^{-1}\Big((\d,s)(\d,s')\Big) = \Phi^{-1}(\d,s)\Phi^{-1}(\d,s') = i_S(s)i_S(s'),
\]
that
\[
i_R(rr') = \Phi^{-1}(rr',\d) = \Phi^{-1}\Big((r,\d)(r',\d)\Big) = \Phi^{-1}(r,\d)\Phi^{-1}(r',\d) = i_R(r)i_R(r'),
\]
and that
\[
p_S(qq') = p_2(\Phi(qq')) = \pi_2(\Phi(q)\Phi(q')) = \pi_2(\Phi(q))\pi_2(\Phi(q'))  = p_S(q) p_S(q').
\]
On the other hand, the mapping $R\rtimes S \to Q$, $(r,s) \mapsto i_R(r)i_S(s)$, becomes $\Phi^{-1}:R\rtimes S \to Q$, whereas the map $Q \to R\rtimes S$, $q\mapsto (p_R(q),p_S(q))$ becomes $\Phi:Q \to R\rtimes S$. Finally, we note also that
\begin{align*}
& p_R\big(\big(i_R(r)i_S(s)\big)\big(i_R(r')i_S(s')\big)\big) = p_R\big(\Phi^{-1}(r,s)\Phi^{-1}(r',s')\big) = \\
&  p_R\big(\Phi^{-1}(r\vp(s,r'),ss')\big) = r\vp(s,r') = p_R(i_R(r))\Big(\big(p_R(i_S(s)i_R(r'))\big)p_R(i_S(s'))\Big) = \\
& \Big(p_R(i_R(r))\big(p_R(i_S(s)i_R(r'))\big)\Big)p_R(i_S(s')).
\end{align*}

Conversely, let $i_S:S \to Q$, $i_R:R \to Q$, and $p_S:Q \to S$ the quasigroup homomorphisms, together with the map $p_R:Q \to R$ satisfying \eqref{p_R-Moufang-I}, such that $\Psi:R\rtimes S \to Q$, $\Psi(r,s) := i_R(r)i_S(s)$, and $\Phi:Q \to R\rtimes S$, $\Phi(q):= (p_R(q),p_S(q))$ are inverse to each other. Thus, the loop structure on $Q$ induces a loop structure on $R\times S$. We shall, furthermore, see that this induced loop structure is in fact one of the form \eqref{semi-direct-multp}. Indeed,
\begin{align*}
& (\d,s)(r',\d) = \Phi(\Psi(\d,s)\Psi(r',\d)) = \Phi\Big(\big(i_R(\d)i_S(s)\big)\big(i_R(r')i_S(\d)\big)\Big) = \Phi\Big(i_S(s)i_R(r')\Big) = \\
& \Big(p_R\big(i_S(s)i_R(r')\big),p_S\big(i_S(s)i_R(r')\big)\Big) =  \Big(p_R\big(i_S(s)i_R(r')\big),p_S\big(i_S(s)\big)p_S\big(i_R(r')\big)\Big) = \\
& \Big(p_R\big(i_S(s)i_R(r')\big),s\Big) = \Big(\vp(s,r'), s\Big),
\end{align*}
where $\vp:S \times R \to R$, $\vp(s,r'):= p_R\big(i_S(s)i_R(r')\big)$. On the  third equality we used the assumption that $i_R,i_S$ are quasigroup homomorphisms, while on the fifth equality we used that of $p_S:Q \to S$ being a quasigroup homomorphism. Finally, on the sixth equality we used $p_S \circ i_S = \Id_S$. Accordingly,
\begin{align*}
& (r,s)(r',s') = \Phi(\Psi(r,s)\Psi(r',s')) = \Phi\Big(\big(i_R(r)i_S(s)\big)\big(i_R(r')i_S(s')\big)\Big) =  \\
& \Big(p_R\big(\big(i_R(r)i_S(s)\big)\big(i_R(r')i_S(s')\big)\big),p_S\big(\big(i_R(r)i_S(s)\big)\big(i_R(r')i_S(s')\big)\big)\Big) =\\
&\Big(p_R(i_R(r))\Big(\big(p_R(i_S(s)i_R(r'))\big)p_R(i_S(s'))\Big),ss'\Big) = \Big(r\vp(s,r'), ss'\Big).
\end{align*}
\end{proof}

If we ask the semi-direct product loop to have the $m$-inverse property, then we have the following more precise result.

\begin{proposition}\label{prop-semi-direct-quasi-IV}
Let $(R,\d)$ be an $m_1$-inverse loop with the permutation $J_R:R \to R$ so that $J_R(\d) = \d$, and that $J_R^{h_1} \in {\rm Aut}(R)$, and $(S,\d)$ is an $m_2$-inverse loop with $J_S:S \to S$ such that $J_S(\d) = \d$, and that $J_S^{h_2} \in {\rm Aut}(S)$. Furthermore, let there be a map $\vp:S \times R \to R$
satisfying
\begin{align}\label{unit-action-QR}
\begin{split}
& \vp(\d,r) = r, \qquad \vp(s,\d) =\d, \\
& \vp(J_S^m(ss'),\vp(J_S^{m+1}(s),r)) = \vp(J_S^m(s'),r), \\
& \vp(s,J_R^m(rr'))\vp(s,J_R^{m+1}(r)) =\vp(s,J_R^m(r')), 
\end{split}
\end{align}
for any $r,r' \in R$, any $s,s' \in S$, and any $m \in \B{Z}$ that satisfies 
\begin{align}\label{cong-eqn-II}
\begin{split}
& m \equiv m_1\,\, ({\rm mod}\,h_1), \\
& m \equiv m_2\,\, ({\rm mod}\,h_2).
\end{split}
\end{align}
Then, $\Big(R \rtimes S := R \times S, (\d,\d)\Big)$ is an $m$-invertible loop with the multiplication
\begin{equation}\label{semi-direct-multp}
(r,s)(r',s') := \Big(r\vp(s,r'),ss'\Big)
\end{equation}
and the permutation 
\begin{equation}\label{perm-semidirect-prod}
J: R \rtimes S \to R \rtimes S, \qquad J(r,s) := (\d,J_S(s))(J_R(r),\d) = \Big(\vp\big(J_S(s),J_R(r)\big), J_S(s) \Big),
\end{equation}
if and only if
\begin{align}\label{m-inverse-cond}
& \begin{cases}
\vp(s,r) = r & \text{\rm if}\,\, m = 2\ell, \\
\vp(J_S^m(ss'),\vp(s,r)) = \vp(J_S^m(s'),r) & \text{\rm if}\,\, m = 2\ell+1,
\end{cases}
\end{align}
for any $s,s' \in S$, and any $r\in R$.
\end{proposition}

\begin{proof}
Assuming the conditions are met, we see at once that
\begin{align*}
& (r,s)J(r,s) = \big[(r,\d)(\d,s)\big]\big[(\d,J_S(s))(J_R(r),\d)\big] = \\
& \big[(r,\d)(\d,s)\big](\vp(J_S(s),J_R(r)),J_S(s)) = \\
& (r,\d)\big[(\d,s)(\vp(J_S(s),J_R(r)),J_S(s))\big] = \\
& (r,\d)(\vp(s,\vp(J_S(s),J_R(r))), sJ_S(s)) = \\
& (r,\d)(J_R(r), \d) = (rJ_R(r),\d) = (\d,\d).
\end{align*}
On the other hand, since
\[
\vp(s,r)J_R(\vp(s,r)) = \d = \vp(s,r)\vp(s,J_R(r)),
\]
we conclude 
\[
J_R(\vp(s,r)) = \vp(s,J_R(r)),
\]
which, in turn, implies that
\begin{align*}
& J\big((\d,s)(r,\d)\big) = J(\vp(s,r),s) = (\d,J_S(s))(J_R(\vp(s,r)),\d) = \\
& (\d,J_S(s))(\vp(s,J_R(r)),\d) = (\vp(J_S(s),\vp(s,J_R(r))), J_S(s))= (J_R(r),J_S(s)),
\end{align*}
and then that
\[
J^m(r,s) = \begin{cases}
(J_R^m(r),J_S^m(s)), & \text{\rm if}\,\, m = 2\ell,\\
(\d,J_S^m(s))(J_R^m(r),\d), & \text{\rm if}\,\, m = 2\ell+1.
\end{cases}
\]
Accordingly, in the case $m = 2\ell$,
\begin{align}\label{m=2l}
\begin{split}
& J^m\big((r,s)(r',s')\big)J^{m+1}(r,s) = J^m(r\vp(s,r'),ss')J^{m+1}(r,s) = \\
& \big[(J_R^m(r\vp(s,r')),\d)(\d, J_S^m(ss'))\big]\big[(\d,J_S^{m+1}(s))(J_R^{m+1}(r),\d)\big] = \\
& (J_R^m(r\vp(s,r')),\d)\Big\{(\d, J_S^m(ss'))\big[(\d,J_S^{m+1}(s))(J_R^{m+1}(r),\d)\big]\Big\} = \\
& (J_R^m(r\vp(s,r')),\d)\big[((\d, J_S^m(ss')))(\vp(J_S^{m+1}(s),J_R^{m+1}(r)),J_S^{m+1}(s))\big] = \\
& (J_R^m(r\vp(s,r')),\d)(\vp(J_S^m(ss'), \vp(J_S^{m+1}(s),J_R^{m+1}(r))),J_S^m(ss')J_S^{m+1}(s)) = \\
& (J_R^m(r\vp(s,r')),\d)(\vp(J_S^m(ss'), \vp(J_S^{m+1}(s),J_R^{m+1}(r))),J_S^m(s')) = \\
& (J_R^m(r\vp(s,r')),\d)(\vp(J_S^m(s'), J_R^{m+1}(r)),J_S^m(s')) = \\
& \Big(J_R^m(r\vp(s,r'))\vp(J_S^m(s'), J_R^{m+1}(r)), J_S^m(s')\Big) = \\
& \Big(J_R^m(r\vp(s,r'))J_R^{m+1}(\vp(J_S^m(s'), r)), J_S^m(s')\Big)  = \big(J_R^m(r'),J_S^m(s')\big) = J^m(r',s')
\end{split}
\end{align}
where; on the sixth equality we used Remark \ref{rk-m+uh}, and that $m\in \B{Z}$ is a solution of the system \eqref{cong-eqn-II}, on the tenth equality we used \eqref{m-inverse-cond}, in addition to Remark \ref{rk-m+uh} and \eqref{cong-eqn-II}. In the case $m = 2\ell+1$,
\begin{align}\label{m=2l+1}
\begin{split}
& J^m\big((r,s)(r',s')\big)J^{m+1}(r,s) = J^m(r\vp(s,r'),ss')J^{m+1}(r,s) = \\
& \big[(\d,J_S^m(ss'))(J_R^m(r\vp(s,r')),\d)\big]\big[(J_R^{m+1}(r),\d)(\d,J_S^{m+1}(s))\big] = \\
& (\vp(J_S^m(ss'), J_R^m(r\vp(s,r'))), J_S^m(ss'))\big[(J_R^{m+1}(r),\d)(\d,J_S^{m+1}(s))\big] = \\
& \big[(\vp(J_S^m(ss'), J_R^m(r\vp(s,r'))), J_S^m(ss'))(J_R^{m+1}(r),\d)\big](\d,J_S^{m+1}(s)) = \\
& (\vp(J_S^m(ss'), J_R^m(r\vp(s,r')))\vp(J_S^m(ss'), J_R^{m+1}(r)), J_S^m(ss'))(\d,J_S^{m+1}(s)) = \\
& (\vp(J_S^m(ss'), J_R^m(\vp(s,r'))), J_S^m(ss'))(\d,J_S^{m+1}(s)) = \\
& (\vp(J_S^m(ss'), J_R^m(\vp(s,r'))), J_S^m(s')) = (J_R^m(\vp(J_S^m(ss'), \vp(s,r'))), J_S^m(s')) = \\
& (J_R^m(\vp(J_S^m(s'), r')), J_S^m(s')) = (\vp(J_S^m(s'), J_R^m(r')), J_S^m(s')) = (\d,J_S^m(s'))(J_R^m(r'),\d) = J^m(r',s')
\end{split}
\end{align}
where; in the sixth equation we used \eqref{unit-action-QR}, in the seventh equation we used Remark \ref{rk-m+uh} and \eqref{cong-eqn-II}, and in the ninth equation we used \eqref{m-inverse-cond}. 

Let, conversely, $R \rtimes S$ be an $m$-inverse loop with the multiplication \eqref{semi-direct-multp} and the permutation \eqref{perm-semidirect-prod}. 

In the case $m = 2\ell$, the tenth equation of \eqref{m=2l} holds, and we have
\[
J_R^m(r\vp(s,r'))J_R^{m+1}(\vp(J_S^m(s'), r)) = J_R^m(r')
\]
for any $r,r' \in R$, and any $s,s' \in S$. In particular, for $r = \d$, we see that
\[
J_R^m(\vp(s,r')) = J_R^m(r'),
\]
and that $\vp(s,r') = r'$, for any $r' \in R$, and any $s\in S$.

In the case $m = 2\ell+1$, however, we have the ninth equation of \eqref{m=2l+1}, that is,
\[
J_R^m(\vp(J_S^m(ss'), \vp(s,r'))) = J_R^m(\vp(J_S^m(s'), r')).
\]
But then, since $J_R:R\to R$ is a permutation, we obtain
\[
\vp(J_S^m(ss'), \vp(s,r')) = \vp(J_S^m(s'), r')
\]
for any $r' \in R$, and any $s,s' \in S$.
\end{proof}

\subsection{Matched pairs of $m$-invertible loops}\label{subsect-matched-pair-quasigr}~

In order to be able to generalize Definition \ref{def-semi-direct} in the presence of two quasigroups, none of which is necessarily the kernel of a quasigroup homomorphism, we adopt the point of view of \cite{BespDrab99,Maji90,Radf85}.

\begin{proposition}\label{prop-matched-pair-quasi}
Let $R$ and $S$ be two loops, with the maps $\vp:S \times R \to R $ and $\psi:S \times R \to S$ satisfying 
\[
\vp(s,\d) = \d, \qquad \vp(\d,r) = r, \qquad \psi(s,\d) = s, \qquad \psi(\d,r) = \d.
\]
Then a loop $Q$ is isomorphic to the loop $R \bowtie S:=R \times S$ with the multiplication given by
\begin{equation}\label{matched-pair-multp}
(r,s)(r',s') := (r\vp(s,r'), \psi(s,r')s'),
\end{equation}
if and only if there are quasigroup homomorphisms $i_S:S \to Q$, $i_R:R \to Q$, together with the maps $p_R:Q \to R$ and $p_S:Q \to S$ satisfying the Moufang-type identities
\begin{align}\label{p_R-Moufang}
\begin{split}
& p_R\big(\big(i_R(r)i_S(s)\big)\big(i_R(r')i_S(s')\big)\big) = p_R(i_R(r))\Big(\big(p_R(i_S(s)i_R(r'))\big)p_R(i_S(s'))\Big) = \\
&\Big(p_R(i_R(r))\big(p_R(i_S(s)i_R(r'))\big)\Big)p_R(i_S(s'))
\end{split}
\end{align}
and
\begin{align}\label{p_S-Moufang}
\begin{split}
& p_S\big(\big(i_R(r)i_S(s)\big)\big(i_R(r')i_S(s')\big)\big) = p_S(i_R(r))\Big(\big(p_S(i_S(s)i_R(r'))\big)p_S(i_S(s'))\Big) = \\
&\Big(p_S(i_R(r))\big(p_S(i_S(s)i_R(r'))\big)\Big)p_S(i_S(s')),
\end{split}
\end{align}
as well as $p_R\circ i_R = \Id_R$ and $p_S\circ i_S = \Id_S$, such that $R\bowtie S \to Q$, $(r,s) \mapsto i_R(r)i_S(s)$ and $Q \to R\bowtie S$, $q\mapsto (p_R(q),p_S(q))$ are inverse to each other.
\end{proposition}

\begin{proof}
Letting $\Phi:Q\to R \bowtie S$ to be the (quasigroup) isomorphism, we consider the mappings  
\[
i_R:R \to Q, \quad i_R(r) := \Phi^{-1}(r,\d), \qquad i_S:S \to Q, \quad i_S(s) := \Phi^{-1}(\d,s)
\]
and
\[
p_R:Q \to R, \quad p_R(q) := \pi_1(\Phi(q)), \qquad p_S:Q \to S, \quad p_S(q) := \pi_2(\Phi(q)),
\]
where $\pi_i$'s denote the projection onto the $i$th component. It is evident that 
\[
(p_R \circ i_R)(r) = \pi_1(r,\d) = r,
\]
for any $r \in R$, as such $p_R \circ i_R = \Id_R$. Similarly, $p_S\circ i_S = \Id_S$. We further see that
\[
i_S(ss') = \Phi^{-1}(\d,ss') = \Phi^{-1}\Big((\d,s)(\d,s')\Big) = \Phi^{-1}(\d,s)\Phi^{-1}(\d,s') = i_S(s)i_S(s'),
\]
and that
\[
i_R(rr') = \Phi^{-1}(rr',\d) = \Phi^{-1}\Big((r,\d)(r',\d)\Big) = \Phi^{-1}(r,\d)\Phi^{-1}(r',\d) = i_R(r)i_R(r').
\]
On the other hand, the mapping $R\bowtie S \to Q$, $(r,s) \mapsto i_R(r)i_S(s)$, becomes $\Phi^{-1}:R\bowtie S \to Q$, whereas the map $Q \to R\bowtie S$, $q\mapsto (p_R(q),p_S(q))$ becomes $\Phi:Q \to R\bowtie S$. Finally, we note also that
\begin{align*}
& p_R\big(\big(i_R(r)i_S(s)\big)\big(i_R(r')i_S(s')\big)\big) = p_R\big(\Phi^{-1}(r,s)\Phi^{-1}(r',s')\big) = \\
&  p_R\big(\Phi^{-1}(r\vp(s,r'),\psi(s,r')s')\big) = r\vp(s,r') = p_R(i_R(r))\Big(\big(p_R(i_S(s)i_R(r'))\big)p_R(i_S(s'))\Big) = \\
& \Big(p_R(i_R(r))\big(p_R(i_S(s)i_R(r'))\big)\Big)p_R(i_S(s')),
\end{align*}
and that, similarly, 
\begin{align*}
& p_S\big(\big(i_R(r)i_S(s)\big)\big(i_R(r')i_S(s')\big)\big) = p_S(i_R(r))\Big(\big(p_S(i_S(s)i_R(r'))\big)p_S(i_S(s'))\Big) = \\
&\Big(p_S(i_R(r))\big(p_S(i_S(s)i_R(r'))\big)\Big)p_S(i_S(s')).
\end{align*}
Conversely, let $i_S:S \to Q$ and $i_R:R \to Q$ be quasigroup homomorphisms, together with the maps $p_R:Q \to R$ and $p_S:Q \to S$ satisfying \eqref{p_R-Moufang} and \eqref{p_S-Moufang}, such that $\Psi:R\bowtie S \to Q$, $\Psi(r,s) := i_R(r)i_S(s)$, and $\Phi:Q \to R\bowtie S$, $\Phi(q):= (p_R(q),p_S(q))$ are inverse to each other. Thus, the loop structure on $Q$ induces a loop structure on $R\times S$. We shall, furthermore, see that this induced loop structure is in fact of the form \eqref{matched-pair-multp}. Indeed,
\begin{align*}
& (\d,s)(r',\d) = \Phi(\Psi(\d,s)\Psi(r',\d)) = \Phi\Big(\big(i_R(\d)i_S(s)\big)\big(i_R(r')i_S(\d)\big)\Big) = \Phi\Big(i_S(s)i_R(r')\Big) = \\
& \Big(p_R\big(i_S(s)i_R(r')\big),p_S\big(i_S(s)i_R(r')\big)\Big) =   \Big(\vp(s,r'), \psi(s,r')\Big),
\end{align*}
where $\vp:S \times R \to R$, $\vp(s,r'):= p_R\big(i_S(s)i_R(r')\big)$, and $\psi:S \times R \to S$, $\psi(s,r'):= p_S\big(i_S(s)i_R(r')\big)$. On the  third equality we used the assumption that $i_R,i_S$ are quasigroup homomorphisms. Accordingly,
\begin{align*}
& (r,s)(r',s') = \Phi(\Psi(r,s)\Psi(r',s')) = \Phi\Big(\big(i_R(r)i_S(s)\big)\big(i_R(r')i_S(s')\big)\Big) = \\
& \Big(p_R\big(\big(i_R(r)i_S(s)\big)\big(i_R(r')i_S(s')\big)\big),p_S\big(\big(i_R(r)i_S(s)\big)\big(i_R(r')i_S(s')\big)\big)\Big) =\\
&\Big(p_R(i_R(r))\Big(\big(p_R(i_S(s)i_R(r'))\big)p_R(i_S(s'))\Big),\Big(p_S(i_R(r))\big(p_S(i_S(s)i_R(r'))\big)\Big)p_S(i_S(s'))\Big) = \\
& \Big(r\vp(s,r'), \psi(s,r')s'\Big).
\end{align*}
\end{proof}

Next, we discuss the matched pair construction for the $m$-inverse property loops.

\begin{proposition}\label{prop-matched-pair-loops}
Let $(R,\d)$ be an $m_1$-inverse loop with the permutation $J_R:R \to R$ so that $J_R(\d) = \d$, and that $J_R^{h_1} \in {\rm Aut}(R)$, and $(S,\d)$ is an $m_2$-inverse loop with $J_S:S \to S$ such that $J_S(\d) = \d$, and that $J_S^{h_2} \in {\rm Aut}(S)$. Furthermore, let there be two maps $\phi:S \times R \to R$ and $\psi:S \times R \to S$
satisfying
\begin{align}
& \phi(\d,r) = r, \quad \phi(s,\d) =\d, \qquad \psi(\d,r) = \d, \quad \psi(s,\d) =s, \label{unit-action-QR-matched-I}\\
& \phi(s,\phi(J_S(s),r)) = r, \label{unit-action-QR-matched-I-I}\\
& \psi(\psi(s, J_R^m(rr')),J_R^{m+1}(r)) = \psi(s,J_R^m(r')), \label{unit-action-QR-matched-II}\\
& \phi(s,J_R^m(rr'))\phi(\psi(s,J_R^m(rr')),J_R^{m+1}(r)) =\phi(s,J_R^m(r')), \label{unit-action-QR-matched-III}\\
& \psi(s,\phi(J_S(s),r))\psi(J_S(s),r) =\d, \label{unit-action-QR-matched-IV}
\end{align}
for any $r,r' \in R$, any $s,s' \in S$, and any $m \in \B{Z}$ that satisfies 
\begin{align}\label{cong-eqn-II-matched}
\begin{split}
& m \equiv m_1\,\, ({\rm mod}\,h_1), \\
& m \equiv m_2\,\, ({\rm mod}\,h_2).
\end{split}
\end{align}
Then, $\Big(R \bowtie S := R \times S, (\d,\d)\Big)$ is an $m$-invertible loop with the multiplication
\begin{equation}\label{matched-pair-multp}
(r,s)(r',s') := \Big(r\phi(s,r'),\psi(s,r')s'\Big)
\end{equation}
and the permutation 
\begin{equation}\label{perm-matched-pair-prod}
J: R \bowtie S \to R \bowtie S, \qquad J(r,s) := (\d,J_S(s))(J_R(r),\d) = \Big(\phi\big(J_S(s),J_R(r)\big),\psi\big(J_S(s),J_R(r)\big) \Big),
\end{equation}
if and only if
\begin{align}\label{m-inverse-cond-matched}
& \begin{cases}
\left. \begin{array}{c}
\phi(s,r) = r, \\ 
\psi(s,r) = s, 
\end{array}\right\}
& \text{\rm if}\,\, m = 2\ell, \\
\left. \begin{array}{c}
\phi(J_S^m(\psi(s,J_R^{-m}(r))s'),\phi(\psi(s,J_R^{-1}(r)), r)) = \phi(J_S^m(s'),r), \\
\Big[\psi(J_S^m(\psi(s,r)s'),J_R^m(\phi(s,r)))\Big]J_S^{m+1}(s) = \psi(J_S^m(s'),J_R^m(r)),
\end{array}\right\}
& \text{\rm if}\,\, m = 2\ell+1,
\end{cases}
\end{align}
for any $s,s' \in S$, and any $r,r'\in R$.
\end{proposition}

\begin{proof}
Assuming the conditions \eqref{m-inverse-cond-matched} are met, we see at once that
\begin{align*}
& (r,s)J(r,s) = \big[(r,\d)(\d,s)\big]\big[(\d,J_S(s))(J_R(r),\d)\big] = \\
& \big[(r,\d)(\d,s)\big](\phi(J_S(s),J_R(r)),\psi(J_S(s),J_R(r))) = \\
& (r,\d)\big[(\d,s)(\phi(J_S(s),J_R(r)),\psi(J_S(s),J_R(r)))\big] = \\
& (r,\d)\Big(\phi(s,\phi(J_S(s),J_R(r))), \psi(s,\phi(J_S(s),J_R(r)))\psi(J_S(s),J_R(r))\Big) = \\
& (r,\d)(J_R(r), \d) = (rJ_R(r),\d) = (\d,\d),
\end{align*}
where on the fifth equality we used \eqref{unit-action-QR-matched-I-I}, and \eqref{unit-action-QR-matched-IV}. Next, in view of \eqref{unit-action-QR-matched-III} and \eqref{unit-action-QR-matched-II}, we have
\begin{align*}
& \big((\d,s)(r,\d)\big)(J_R(r),J_S(s)) = (\phi(s,r),\psi(s,r))(J_R(r),J_S(s)) = \\
& \big(\phi(s,r)\phi(\psi(s,r),J_R(r)), \psi(\psi(s,r),J_R(r))J_S(s)\big) = (\d,\d),
\end{align*}
which implies that
\[
J\big((\d,s)(r,\d)\big) = (J_R(r),J_S(s)).
\]
On the other hand, in view of \eqref{unit-action-QR-matched-III} we have
\[
\phi(s,r)J_R(\phi(s,r)) = \d = \phi(s,r)\phi(\psi(s,r),J_R(r)),
\]
and hence we conclude 
\begin{equation}\label{incerse-of-left-action}
J_R(\phi(s,r)) = \phi(\psi(s,r),J_R(r)).
\end{equation}
Let us note further that \eqref{incerse-of-left-action}, together with \eqref{unit-action-QR-matched-II}, implies 
\[
J_R^m(\phi(s,r)) = \begin{cases}
\phi(s,J^m_R(r)) & \text{if  } m = 2\ell, \\
\phi(\psi(s,J^{m-1}_R(r)),J^m_R(r)) & \text{if  } m=2\ell+1.
\end{cases}
\]
Accordingly, in the case $m = 2\ell$,
\begin{align}\label{m=2l-matched}
\begin{split}
& J^m\big((r,s)(r',s')\big)J^{m+1}(r,s) = J^m(r\phi(s,r'),\psi(s,r')s')J^{m+1}(r,s) = \\
& \big[(J_R^m(r\phi(s,r')),\d)(\d, J_S^m(\psi(s,r')s'))\big]\big[(\d,J_S^{m+1}(s))(J_R^{m+1}(r),\d)\big] = \\
& (J_R^m(r\phi(s,r')),\d)\Big\{(\d, J_S^m(\psi(s,r')s'))\big[(\d,J_S^{m+1}(s))(J_R^{m+1}(r),\d)\big]\Big\} = \\
& (J_R^m(r\phi(s,r')),\d)\big[((\d, J_S^m(\psi(s,r')s')))(\phi(J_S^{m+1}(s),J_R^{m+1}(r)),\psi(J_S^{m+1}(s),J_R^{m+1}(r)))\big] = \\
& (J_R^m(r\phi(s,r')),\d)\Big(\phi(J_S^m(\psi(s,r')s'), \phi(J_S^{m+1}(s),J_R^{m+1}(r))), \\
& \hspace{1cm}\psi(J_S^m(\psi(s,r')s'), \phi(J_S^{m+1}(s),J_R^{m+1}(r)))\psi(J_S^{m+1}(s),J_R^{m+1}(r))\Big) =\\
& \Big(J_R^m(r\phi(s,r'))\phi(J_S^m(\psi(s,r')s'), \phi(J_S^{m+1}(s),J_R^{m+1}(r))), \\
& \hspace{1cm}\psi(J_S^m(\psi(s,r')s'), \phi(J_S^{m+1}(s),J_R^{m+1}(r)))\psi(J_S^{m+1}(s),J_R^{m+1}(r))\Big) = \\
& \big(J_R^m(r'),J_S^m(s')\big) = J^m(r',s'),
\end{split}
\end{align}
where; on the seventh equality we used \eqref{m-inverse-cond-matched}. In the case $m = 2\ell+1$,
\begin{align}\label{m=2l+1-matched}
\begin{split}
& J^m\big((r,s)(r',s')\big)J^{m+1}(r,s) = J^m(r\phi(s,r'),\psi(s,r')s')J^{m+1}(r,s) = \\
& \big[(\d,J_S^m(\psi(s,r')s'))(J_R^m(r\phi(s,r')),\d)\big]\big[(J_R^{m+1}(r),\d)(\d,J_S^{m+1}(s))\big] = \\
& (\phi(J_S^m(\psi(s,r')s'), J_R^m(r\phi(s,r'))),\psi(J_S^m(\psi(s,r')s'), J_R^m(r\phi(s,r'))))\big[(J_R^{m+1}(r),\d)(\d,J_S^{m+1}(s))\big] = \\
& \big[(\phi(J_S^m(\psi(s,r')s'), J_R^m(r\phi(s,r'))),\psi(J_S^m(\psi(s,r')s'), J_R^m(r\phi(s,r'))))(J_R^{m+1}(r),\d)\big] (\d,J_S^{m+1}(s))= \\
& \Big[\Big(\phi(J_S^m(\psi(s,r')s'), J_R^m(r\phi(s,r'))) \phi(\psi(J_S^m(\psi(s,r')s'), J_R^m(r\phi(s,r'))), J_R^{m+1}(r)), \\
&\hspace{6.5cm} \psi(\psi(J_S^m(\psi(s,r')s'), J_R^m(r\phi(s,r'))), J_R^{m+1}(r))\Big)\Big](\d,J_S^{m+1}(s)) = \\
& \Big(\phi(J_S^m(\psi(s,r')s'), J_R^m(\phi(s,r'))), \psi(J_S^m(\psi(s,r')s'), J_R^m(\phi(s,r')))\Big)(\d,J_S^{m+1}(s)) =\\
& \Big(\phi(J_S^m(\psi(s,r')s'), J_R^m(\phi(s,r'))), \psi(J_S^m(\psi(s,r')s'), J_R^m(\phi(s,r')))J_S^{m+1}(s)\Big) =\\
& \Big(\phi(J_S^m(\psi(s,r')s'), [\phi(\psi(s,J_R^{m-1}(r')),J_R^m(r'))]), \psi(J_S^m(\psi(s,r')s'), J_R^m(\phi(s,r')))J_S^{m+1}(s)\Big) = \\
& \Big(\phi(J_S^m(s'),J_R^m(r')), \psi(J_S^m(s'),J_R^m(r'))\Big) = (\d,J_S^m(s'))(J_R^m(r'),\d) = J^m(r',s')
\end{split}
\end{align}
where; in the sixth equation we used \eqref{unit-action-QR-matched-III} and the second identity of \eqref{unit-action-QR-matched-II}, on the eighth equation we used \eqref{incerse-of-left-action}, and finally on the ninth equation we used (both identities of) \eqref{m-inverse-cond-matched}, in addition to Remark \ref{rk-m+uh} and \eqref{cong-eqn-II-matched}. 

Let, conversely, $R \bowtie S$ be an $m$-inverse loop with the multiplication \eqref{matched-pair-multp} and the permutation \eqref{perm-matched-pair-prod}. 

In the case $m = 2\ell$, the seventh equation of \eqref{m=2l-matched} holds, and we have
\[
J_R^m(r\phi(s,r'))\phi(J_S^m(\psi(s,r')s'), \phi(J_S^{m+1}(s),J_R^{m+1}(r))) = J_R^m(r')
\]
together with
\[
\psi(J_S^m(\psi(s,r')s'), \phi(J_S^{m+1}(s),J_R^{m+1}(r)))\psi(J_S^{m+1}(s),J_R^{m+1}(r)) = J_S^m(s')
\]
for any $r,r' \in R$, and any $s,s' \in S$. In particular, for $r = \d$, the former equality yields
\[
J_R^m(\vp(s,r')) = J_R^m(r'),
\]
hence $\vp(s,r') = r'$, for any $r' \in R$, and any $s\in S$. For, on the other hand, $s = \d$, the latter results in
\[
\psi(J_S^m(s), J_R^{m+1}(r)) =  J_S^m(s).
\]
Once again, in view of the fact that $J_R:R\to R$ and $J_S:S \to S$ are both permutations, we deduce that $\psi(s,r) = s$ for any $r \in R$ and any $s\in S$. 

In the case $m = 2\ell+1$, however, the ninth equation of \eqref{m=2l+1-matched} holds, that is,
\[
\phi(J_S^m(\psi(s,r')s'), [\phi(\psi(s,J_R^{m-1}(r')),J_R^m(r'))]) = \phi(J_S^m(s'),J_R^m(r')),
\]
and
\[
\psi(J_S^m(\psi(s,r')s'), J_R^m(\phi(s,r')))J_S^{m+1}(s) = \psi(J_S^m(s),J_R^m(r')).
\]
The latter is nothing but the second identity of \eqref{m-inverse-cond-matched}, whereas the first identity of  \eqref{m-inverse-cond-matched} is obtained by taking $r' = J_R^{-m}(r)$ in the former.
\end{proof}

\begin{definition}
Let $(R,J_R,\d_R)$ be an $m_1$-inverse property loop such that $J_R(\d_R) = \d_R$, and that $J_R^{h_1} \in {\rm Aut}(R)$, and $(S,J_S,\d_S)$ be an $m_2$-inverse property loop such that $J_S(\d_S) = \d_S$, and that $J_S^{h_2} \in {\rm Aut}(S)$. Let also $m\in \B{Z}$ be a solution of 
\begin{align*}
& m \equiv m_1\,\, ({\rm mod}\,h_1), \\
& m \equiv m_2\,\, ({\rm mod}\,h_2).
\end{align*} 
Then, $(R,S)$ is called a ``matched pair of $m$-inverse property loops'' if $(R,J_R,\d_R)$ and $(S,J_S,\d_S)$ satisfy the conditions \eqref{unit-action-QR-matched-I}-\eqref{unit-action-QR-matched-IV}.
\end{definition}

\begin{remark}
We see that if $(R,S)$ is a matched pair of $m$-inverse property quasigroups, then $R\bowtie S:= R \times S$ is an $m$-inverse property quasigroup if and only if \eqref{m-inverse-cond-matched} holds. From the point of view of the generalization of groups, this is a manifestation of the fact that any group may be considered as an odd-inverse property quasigroup, while only commuttative groups fall into the category of even-inverse property quasigroups. Furthermore, we already know from the theory of matched pairs (of groups) that the matched pair group is commutative if and only if the mutual actions are trivial. 
\end{remark}

The following is an analogue of \cite[Prop. 6.2.15]{Majid-book}.

\begin{proposition}\label{prop-matched-pair-quasi-universal}
Let $(R,\d)$ be an $m_1$-inverse loop with the permutation $J_R:R \to R$ so that $J_R(\d) = \d$, and that $J_R^{h_1} \in {\rm Aut}(R)$, and $(S,\d)$ is an $m_2$-inverse loop with $J_S:S \to S$ such that $J_S(\d) = \d$, and that $J_S^{h_2} \in {\rm Aut}(S)$. Let also $m\in \B{Z}$ be a solution of 
\begin{align*}
& m \equiv m_1\,\, ({\rm mod}\,h_1), \\
& m \equiv m_2\,\, ({\rm mod}\,h_2),
\end{align*} 
and $(Q,\d)$ be an $m$-inverse loop so that $(R,\d)$ is an $m_1$-inverse subloop of $(Q,\d)$, and $(S,\d)$ is an $m_2$-inverse subloop of $(Q,\d)$;
\[
(R,\d) \xhookrightarrow{} (Q,\d) \xhookleftarrow{} (S,\d),
\]
that the multiplication in $Q$ yields an isomorphism
\begin{equation}\label{Theta-map}
\Theta:R\times S \to Q, \qquad (r,s) \mapsto rs,
\end{equation}
under which the multiplications are compatible as
\begin{equation}\label{compatibilities}
(rs)q = r(sq), \qquad q(rs) = (qr)s,
\end{equation}
and the inversions as 
\begin{equation}\label{J-on-Q}
J_Q(rs) = J_S(s)J_R(r), \qquad J_Q(sr) = J_R(r)J_S(s)
\end{equation}
for any $r\in R$, any $s \in S$, and any $q\in Q$. Then, $(R,S)$ is a matched pair of $m$-inverse loops, and $Q \cong R\bowtie S$ as quasigroups.
\end{proposition}

\begin{proof}
Let us begin with the mappings
\begin{equation}\label{phi-and-psi}
\phi:S\times R \to R, \qquad \psi:S\times R \to S
\end{equation}
given by
\[
\phi(s,r) := (\pi_1\circ \Theta^{-1})(sr), \qquad \psi(s,r) := (\pi_2\circ \Theta^{-1})(sr),
\]
where
\[
\pi_1:R\times S \to R, \qquad \pi_2:R\times S \to S
\]
are the projections onto the first and the second component respectively. It then follows at once that
\begin{equation}\label{multp-in-Q-and-RS}
sr = \Theta\big(\phi(s,r),\psi(s,r)\big) = \Theta((\d,s)(r,\d)),
\end{equation}
that is, the isomorphism \eqref{Theta-map} respect the multiplications in $Q$ and $R\bowtie S$.

It remains to show that the mappings \eqref{phi-and-psi} have the properties \eqref{unit-action-QR-matched-I}-\eqref{unit-action-QR-matched-IV}.

The first one, \eqref{unit-action-QR-matched-I}, follows from the consideration of $r=\d$ and $s=\d$ in \eqref{multp-in-Q-and-RS}, respectively.

Next, in view of \eqref{J-on-Q} the property $qJ_Q(q) = \d$ implies $(rs)J_Q(rs) = \d$ for any $r\in R$ and any $s\in S$, which in turn implies \eqref{unit-action-QR-matched-I-I} and \eqref{unit-action-QR-matched-IV}.

On the other hand, \eqref{compatibilities}, and $J^m_Q(qq')J^{m+1}_Q(q) = J^m_Q(q')$ for any $q,q'\in Q$ yields, along the lines of \eqref{m=2l+1-matched},
\begin{align*}
& \Big[\Big(\phi(J_S^m(\psi(s,r')s'), J_R^m(r\phi(s,r'))) \phi(\psi(J_S^m(\psi(s,r')s'), J_R^m(r\phi(s,r'))), J_R^{m+1}(r)), \\
&\hspace{6.5cm} \psi(\psi(J_S^m(\psi(s,r')s'), J_R^m(r\phi(s,r'))), J_R^{m+1}(r))\Big)\Big](\d,J_S^{m+1}(s)) = \\
& \Big(\phi(J_S^m(s'),J_R^m(r')), \psi(J_S^m(s'),J_R^m(r'))\Big). \end{align*}
In particular, for $s=\d$ then we see that
\begin{align*}
& \Big[\Big(\phi(J_S^m(s'), J_R^m(rr')) \phi(\psi(J_S^m(s'), J_R^m(rr')), J_R^{m+1}(r)), \\
&\hspace{6.5cm} \psi(\psi(J_S^m(s'), J_R^m(rr')), J_R^{m+1}(r))\Big)\Big] = \\
& \Big(\phi(J_S^m(s'),J_R^m(r')), \psi(J_S^m(s'),J_R^m(r'))\Big),
\end{align*}
which is equivalent to \eqref{unit-action-QR-matched-II} and \eqref{unit-action-QR-matched-III}.

Finally, having obtained \eqref{unit-action-QR-matched-I}-\eqref{unit-action-QR-matched-IV}, the condition \eqref{m-inverse-cond-matched} follows from the ninth equality of \eqref{m=2l+1-matched} in the odd case, while it is a result of the seventh equality of \eqref{m=2l-matched} in the even case.
\end{proof}

Let us illustrate with an example.

\begin{example}
Given a matched pair of groups $(G,H)$, and two abelian groups $V$ and $W$, let
\begin{equation}\label{sigma-map}
\Lambda:(G\bowtie H)\times (G\bowtie H) \to V\times W, \qquad \Lambda\Big((x,y), (x',y')\Big) := \Big(\vp(x,x'), \chi(y,y')\Big),
\end{equation}
for such $\vp:G\times G\to V$ and $\chi:H\times H \to W$ that
\begin{equation}\label{vp-invariance}
\vp(x,x') = \vp(x,y\rt x')
\end{equation}
and 
\begin{equation}\label{chi-invariance}
\chi(y,y') = \vp(y\lt x,y')
\end{equation}
for any $x,x'\in G$, and any $y,y'\in H$. Then let $(G\bowtie H) \times_\Lambda (V\times W)$ be the $(2\ell+1)$-invertible loop of Example \ref{ex-odd-invtble-loop}. As such, \eqref{sigma-map} satisfies \eqref{quasi-0}, and we obtain
\[
\vp(1,x) = 0 = \vp(x,1), \qquad \chi(1,y) = 0 = \chi(y,1)
\]
for any $(x,y) \in G\times H$. 

Similarly, imposing \eqref{quasi-I} onto \eqref{sigma-map}, 
\[
\Lambda\Big((x,y), (x,y)^{-1}\Big) = \Lambda\Big((x,y), (y^{-1}\rt x^{-1}, y^{-1} \lt x^{-1})\Big) = 0
\]
 we obtain
 \[
 \vp(x,y^{-1}\rt x^{-1}) = 0, \qquad \chi(y,y^{-1} \lt x^{-1}) = 0
 \]
 for any $(x,y) \in G\times H$. In particular, for $y=1 \in H$ we obtain
\[
\vp(x,x^{-1}) = 0,
\]
for any $x\in G$, and setting $x=1\in G$ we arrive at
\[
\chi(y,y^{-1}) = 0,
\]
for any $y\in H$.

Finally, since \eqref{sigma-map} is bound to satisfy \eqref{quasi-II}, that is,
\[
\Lambda\Big((x',y')^{-1}(x,y)^{-1}, (x,y)\Big) = \Lambda\Big((x,y),(x',y')\Big),
\]
for any $(x,y), (x',y') \in G\times H$, or equivalently
\[
\Lambda\Big((x',y')(x,y), (x,y)^{-1}\Big) = \Lambda\Big((x,y)^{-1},(x',y')^{-1}\Big),
\]
we have
\[
\Lambda\Big((x'(y'\rt x), (y'\lt x)y), (y^{-1} \rt x^{-1},y^{-1} \lt x^{-1})\Big) = \Lambda\Big((y^{-1} \rt x^{-1},y^{-1} \lt x^{-1}),(y'^{-1} \rt x'^{-1},y'^{-1} \lt x'^{-1})\Big),
\]
that is,
\[
\vp(x'(y'\rt x), y^{-1} \rt x^{-1}) = \vp(y^{-1} \rt x^{-1}, y'^{-1} \rt x'^{-1})
\]
and
\[
\chi((y'\lt x)y, y^{-1} \lt x^{-1}) = \chi(y^{-1} \lt x^{-1}, y'^{-1} \lt x'^{-1})
\]
for any $x,x'\in G$, and any $y,y'\in H$. Now $y=y'=1\in H$ (resp. $x=x'=1 \in G$) leads to 
\[
\vp(x'x, x^{-1}) = \vp(x^{-1}, x'^{-1}) \qquad ({\rm resp.}\,\,\, \chi(y'y, y^{-1}) = \chi(y^{-1}, y'^{-1})).
\]
As a result, we have the $(2\ell_1+1)$-invertible loop $G\times_\vp V$, and the $(2\ell_2+1)$-invertible loop $H\times_\chi W$, for any $\ell_1,\ell_2\in \B{Z}$, in such a way that
\[
G\times_\vp V \to (G\bowtie H) \times_\Lambda (V\times W), \qquad (x,v) \mapsto \Big((x,1),(v,0)\Big)
\]
and
\[
H\times_\chi W \to (G\bowtie H) \times_\Lambda (V\times W), \qquad (y,w) \mapsto \Big((1,y),(0,w)\Big)
\] 
are quasigroup homomorphisms. Moreover, the multiplication in $(G\bowtie H) \times_\Lambda (V\times W)$ yields the isomorphism
\begin{align*}
& \Theta: (G\times_\vp V) \times (H\times_\chi W) \to (G\bowtie H) \times_\Lambda (V\times W), \\
& \Big((x,v),(y,w)\Big) \mapsto \Big((x,y),(v,w)\Big).
\end{align*}
Let us finally show that \eqref{compatibilities} and \eqref{J-on-Q} are satisfied. As for the former, we simply observe for any $(x,v) \in G\times_\vp V$, any $(y,w) \in H\times_\chi W$, and any $((x',y'),(v',w')) \in (G\bowtie H) \times_\Lambda (V\times W)$,
\begin{align*}
& \Big[(x,v)(y,w)\Big]\Big((x',y'),(v',w')\Big) = \\
& \Big[\Big((x,1),(v,0)\Big)\Big((1,y),(0,w)\Big)\Big]\Big((x',y'),(v',w')\Big) = \\
& \Big((x,y),(v,w)\Big)\Big((x',y'),(v',w')\Big) = \\
& \Big((x(y\rt x'), (y\lt x')y'),(\vp(x,x')+v+v', \chi(y,y')+w+w')\Big) = \\
& \Big((x(y\rt x'), (y\lt x')y'),(\vp(x,y\rt x')+v+v', \chi(y,y')+w+w')\Big) = \\
& \Big((x,1),(v,0)\Big)\Big((y\rt x',(y\rt x')y'),(v',\chi(y,y')+w+w')\Big) = \\
& \Big((x,1),(v,0)\Big)\Big[\Big((1,y),(0,w)\Big)\Big((x',y'),(v',w')\Big)\Big] = \\
& (x,v)\Big[(y,w)\Big((x',y'),(v',w')\Big)\Big],
\end{align*}
where we used \eqref{vp-invariance} in the fourth equality. Similarly, \eqref{chi-invariance} yields 
\[
\Big((x,y),(v,w)\Big)\Big[(x',v')(y',w')\Big] = \Big[\Big((x,y),(v,w)\Big)(x',v')\Big](y',w').
\]
Accordingly, \eqref{compatibilities} holds. As for \eqref{J-on-Q}, we do note that
\begin{align*}
& J\Big((x,v)(y,w)\Big) = J\Big((x,y),(v,w)\Big) = \Big((x,y)^{-1},(-v,-w)\Big) = \\
& \Big((y^{-1}\rt x^{-1},y^{-1}\lt x^{-1}),(v,w)\Big) =  \Big((1,y^{-1}),(0,-w)\Big)\Big((x^{-1},1),(-v,0)\Big) = \\
& (y^{-1},-w)(x^{-1},-v) = J_{H\times_\chi W}(y,w)J_{G\times_\vp V}(x,v),
\end{align*}
and that
\begin{align*}
& J\Big((y,w)(x,v)\Big) = J\Big((y\rt x, y\lt x),(v,w)\Big) = \Big((y\rt x, y\lt x)^{-1},(-v,-w)\Big) = \\
& \Big((x^{-1},y^{-1}),(-v,-w)\Big) = \Big((x^{-1},1),(-v,0)\Big)\Big((1,y^{-1}),(0,-w)\Big) = \\
& (x^{-1},-v)(y^{-1},-w) = J_{G\times_\vp V}(x,v)J_{H\times_\chi W}(y,w).
\end{align*}
We may now say that the hypotheses of Proposition \ref{prop-matched-pair-quasi-universal} hold with $R:= G\times_\vp V$, $S:= H\times_\chi W$, $Q:= (G\bowtie H)\times_\Lambda (V\times W)$, $m:=2\ell +1$, $m_1:=2\ell_1+1$, $m_2:=2\ell_2+1$, for any $\ell,\ell_1,\ell_2\in \B{Z}$, and $h_1 = 2 = h_2$, that $(G\times_\vp V, H\times_\chi W)$ is a matched pair of $(2\ell +1)$-invertible loops, and that
\[
(G\bowtie H)\times_\Lambda (V\times W) \cong (G\times_\vp V) \bowtie (H\times_\chi W).
\]
Indeed, the mutual  ``actions'' 
\[
\phi:(H\times_\chi W) \times (G\times_\vp V) \to (G\times_\vp V), \qquad \Big((y,w),(x,v)\Big) \mapsto (y\rt x, v) 
\]
and
\[
\psi:(H\times_\chi W) \times (G\times_\vp V) \to (H\times_\chi W), \qquad \Big((y,w),(x,v)\Big) \mapsto (y\lt x, w) 
\]
which fit (in view of \eqref{vp-invariance} and \eqref{chi-invariance}) into
\begin{align*}
&\Big((x,v);(y,w)\Big)\Big((x',v');(y',w')\Big)  =\Big((x,y),(v,w)\Big)\Big((x',y'),(v',w')\Big) = \\
&\Big((x,v)\phi\Big((y,w),(x',v')\Big); \psi\Big((y,w),(x',v')\Big)(y',w')\Big)
\end{align*}
satisfy the compatibilities \eqref{unit-action-QR-matched-I}-\eqref{unit-action-QR-matched-IV}, as well as \eqref{m-inverse-cond-matched}, merely from the matched pair compatibilities for groups.
\end{example}

\section{Linearizations}\label{sect-linearization}

Following the terminology and the point of view of \cite{KlimMaji10-II,KlimMaji10}, we shall consider the Hopf analogues of the $m$-inverse property loops, under the name ``$m$-invertible Hopf quasigroup''.

\subsection{$m$-invertible Hopf quasigroups}\label{subsect-m-invtb-Hopf-quasigr}~

Along the lines of \cite[Def. 4.1]{KlimMaji10}, see also \cite[Def. 2.1]{KlimMaji10-II}, we now introduce what we call an ``$m$-inverse property Hopf quasigroup''.

\begin{definition}\label{def-m-inv-loop}
Let $\C{H}$ be a $k$-linear space equipped with the linear maps $\mu:\C{H}\ot \C{H}\to \C{H}$, $\eta:k\to \C{H}$, $\D:\C{H} \to \C{H}\ot \C{H}$, $\ve:\C{H} \to k$, and $S:\C{H}\to \C{H}$. Then, the six-tuple $(\C{H},\mu,\eta,\D,\ve,S)$ is called an $m$-inverse property Hopf quasigroup if 
\begin{itemize}
\item[(i)] $(\C{H},\mu,\eta)$ is a unital, not-necessarily associative algebra,
\item[(ii)] $(\C{H},\D,\ve)$ is a coassociative and counital coalgebra,
\item[(iii)] $\D:\C{H} \to \C{H}\ot \C{H}$ and $\ve:\C{H} \to k$ are multiplicative,
\item[(iv)] $S:\C{H}\to \C{H}$ is the unique coalgebra anti-automorphism satisfying
\begin{equation}\label{S-prop}
h\ns{1}S(h\ns{2}) = \ve(h)\d = S(h\ns{1})h\ns{2},
\end{equation}
so that
\begin{equation}\label{S-m-prop}
S^m(h\ns{2}g)S^{m+1}(h\ns{1}) = \ve(h)S^m(g)
\end{equation}
holds for any $h,g\in \C{H}$.
\end{itemize}
\end{definition}

\begin{example}
Let $(Q,\d,J)$ be an $m$-inverse property loop. Then the linear space $kQ$ is a $m$-inverse property Hopf quasigroup via 
\begin{itemize}
\item[(i)] the multiplication
\[
\mu:kQ\ot kQ \to kQ, \qquad \mu(q,q') := qq',
\]
defined as a linear extension of the multiplication on $Q$, the unit
\[
\eta:k\to kQ, \qquad \eta(\a) := \a\d,
\]
\item[(ii)] the comultiplication
\[
\D:kQ\to kQ\ot kQ, \qquad \D(q) := q\ot q
\]
as the linear extension of the diagonal map, the counit
\[
\ve:kQ\to k, \qquad \ve(q) = 1,
\]
\item[(iii)] and the ``antipode''
\[
S:kQ\to kQ, \qquad S(q) := J(q).
\]
\end{itemize}
\end{example}

The following adaptation of \cite[Rk. 2.2]{KeedShch03} will be instrumental in the construction of the products of Hopf quasigroups.

\begin{remark}\label{Rk-tensor-prod-Hopf-quasigr}
Let $(\C{H},\mu,\eta,\D,\ve,S)$ be an $m$-inverse property Hopf quasigroup such that $S^r\in {\rm Aut}(\C{H})$, \ie $S^r(hg) = S^r(h)S^r(g)$, and $\D(S^r(h)) = S^r(h\ns{1}) \ot S^r(h\ns{2})$, for any $h \in \C{H}$. Then, $(\C{H},\mu,\eta,\D,\ve,S)$ be an $(m+ur)$-inverse property Hopf quasigroup for any $u\in \B{Z}$. 

Indeed, 
\begin{align*}
& S^{m+ur}(h\ns{2}g)S^{m+1+ur}(h\ns{1}) = S^m(S^{ur}(h\ns{2})S^{ur}(g))S^{m+1}(S^{ur}(h\ns{1})) = \\
& S^m(S^{ur}(h)\ns{2}S^{ur}(g))S^{m+1}(S^{ur}(h)\ns{1}) = S^m(S^{ur}(g)) = S^{m+ur}(g).
\end{align*}
\end{remark}

\subsection{Matched pairs of $m$-inverse property Hopf quasigroups}\label{subsect-matched-pair-Hopf-quasi}~

For convenience, let us begin with the tensor product Hopf quasigroups. More precisely, the following result is the Hopf counterpart of  \cite[Thm. 5.1]{KeedShch03}, that is, Theorem \ref{thm-quasigr-direct-prod} above.

\begin{theorem}
Let $(\C{H}_1,\mu_1,\eta_1,\D_1,\ve_1,S_1)$ be an $m_1$-inverse Hopf quasigroup so that $S_1^{h_1} \in {\rm Aut}(\C{H}_1)$, and $(\C{H}_2,\mu_2,\eta_2,\D_2,\ve_2,S_2)$ be an $m_2$-inverse Hopf quasigroup such that $S_2^{h_2} \in {\rm Aut}(\C{H}_2)$. Then $\C{H}_1 \ot \C{H}_2$ is an $m$-inverse quasigroup with the tensor product structure maps, for any $m \in \B{Z}$ that satisfies 
\begin{align}\label{cong-eqn-Hopf}
\begin{split}
& m \equiv m_1\,\, ({\rm mod}\,h_1), \\
& m \equiv m_2\,\, ({\rm mod}\,h_2).
\end{split}
\end{align}
\end{theorem}

\begin{proof}
It follows at once that
\begin{itemize}
\item[(i)] $(\C{H}_1\ot \C{H}_2,\mu_\ot,\eta_\ot)$ is a (not necessarily associative) unital algebra via
\begin{align*}
& \mu_\ot:= (\mu_1\ot \mu_2)\circ(\Id\ot \tau\ot \Id):(\C{H}_1\ot \C{H}_2) \ot (\C{H}_1\ot \C{H}_2) \to \C{H}_1\ot \C{H}_2, \\
& \mu_\ot\Big((h\ot h')\ot (g\ot g')\Big) := \mu_1(h\ot g) \ot \mu_2(h'\ot g')
\end{align*}
and
\[
\eta_\ot:=\eta_1\ot \eta_2: k \to \C{H}_1\ot \C{H}_2, \qquad \eta_\ot(\a) := \a\eta_1(1)\ot \eta_2(1),
\]
\item[(ii)] $(\C{H}_1\ot \C{H}_2,(\Id\ot \tau\ot \Id)\circ(\D_1\ot \D_2),\ve_1\ot \ve_2)$ is a coassociative counital coalgebra, such that
\item[(iii)] the coalgebra structure maps
\begin{align*}
& \D_\ot:=(\Id\ot \tau\ot \Id)\circ(\D_1\ot \D_2): \C{H}_1\ot \C{H}_2 \to (\C{H}_1\ot \C{H}_2) \ot (\C{H}_1\ot \C{H}_2), \\
& \D_\ot(h\ot h') = (h\ns{1} \ot h'\ns{1}) \ot (h\ns{2} \ot h'\ns{2})
\end{align*}
and 
\[
\ve_\ot:=\ve_1\ot \ve_2: \C{H}_1\ot \C{H}_2 \to k, \qquad \ve_\ot(h\ot h') = \ve_1(h)\ve_2(h')
\]
are multiplicative.
\item[(iv)] Finally, in view of Remark \ref{Rk-tensor-prod-Hopf-quasigr} above, for any solution $m\in \B{Z}$ of \eqref{cong-eqn-Hopf}
\begin{align*}
& (S_1\ot S_2)^m((h\ns{2}\ot h'\ns{2})(g\ot g'))(S_1\ot S_2)^{m+1}(h\ns{1}\ot h'\ns{1}) = \\
& S_1^m(h\ns{2}g)S_1^{m+1}(h\ns{1}) \ot S_2^m(h'\ns{2}g')S_2^{m+1}(h'\ns{1}) = S_1^m(g)\ot S_2^m(g') = (S_1\ot S_2)^m(g\ot g').
\end{align*}
\end{itemize}
\end{proof}

As for the matched pair construction, Proposition \ref{prop-matched-pair-loops} upgrades to the following proposition. However, we shall first need a technical lemma.

\begin{lemma}\label{lemma-antipode}
Let $(\C{H}_1,\mu_1,\eta_1,\D_1,\ve_1,S_1)$ be an $m_1$-inverse Hopf quasigroup, and $(\C{H}_2,\mu_2,\eta_2,\D_2,\ve_2,S_2)$ be an $m_2$-inverse Hopf quasigroup. Moreover, let there be two maps $\phi:\C{H}_2 \ot \C{H}_1 \to \C{H}_1$ and $\psi:\C{H}_2 \ot \C{H}_1 \to \C{H}_2$
satisfying
\begin{align}
& \phi(S(h'\ns{1}),\phi(h'\ns{2},h)) = \ve_2(h')h = \phi(h'\ns{1},\phi(S(h'\ns{2}),h)) , \label{unit-action-QR-matched-II-Hopf-a'-lemma}\\
& \psi(\psi(h',S(h\ns{1})),h\ns{2}) = \ve_1(h)h' = \psi(\psi(h',h\ns{1}),S(h\ns{2})) \label{unit-action-QR-matched-II-Hopf-b'-lemma}\\
& \D_1(\phi(h',h)) = \phi(h'\ns{1},h\ns{1}) \ot \phi(h'\ns{2},h\ns{2}), \qquad \ve_1(\phi(h',h)) = \ve_1(h)\ve_2(h'), \label{module-coalg-I-lemma}\\
& \D_2(\psi(h',h)) = \psi(h'\ns{1},h\ns{1}) \ot \psi(h'\ns{2},h\ns{2}), \qquad \ve_2(\psi(h',h)) = \ve_1(h)\ve_2(h'), \label{module-coalg-II-lemma}\\
& \phi(h'\ns{1},S(h\ns{2}))\big[\phi(\psi(h'\ns{2},S(h\ns{1})),h\ns{3})\big] = \ve_1(h)\ve_2(h') =  \phi(h'\ns{1}, h\ns{1})\big[\phi(\psi(h'\ns{2},h\ns{2}),S(h\ns{3}))\big], \label{unit-action-QR-matched-III-Hopf-a-lemma}\\
& \big[\psi(S(h'\ns{1}),\phi(h'\ns{2},h\ns{1}))\big]\psi(h'\ns{3},h\ns{2}) = \ve_1(h)\ve_2(h') = \label{unit-action-QR-matched-IV-Hopf-a-lemma}\\
& \hspace{6cm} \big[\psi(h'\ns{1},\phi(S(h'\ns{3}),h\ns{1}))\big]\psi(S(h'\ns{2}),h\ns{2}), \notag\\
& \psi(h'\ns{1},h\ns{1}) \ot \phi(h'\ns{2},h\ns{2}) = \psi(h'\ns{2},h\ns{2}) \ot \phi(h'\ns{1},h\ns{1})  \label{unit-action-QR-matched-V-Hopf-lemma}
\end{align}
for any $h,g \in \C{H}_1$, any $h',g' \in \C{H}_2$. Then the mapping
\begin{align}\label{perm-matched-pair-prod-Hopf-lemma}
\begin{split}
& S_{\bowtie}: \C{H}_1 \ot \C{H}_2 \to \C{H}_1 \ot \C{H}_2, \\ 
& S_{\bowtie}(h\ot h') := (\d_1\ot S_2(h'))(S_1(h)\ot \d_2) = \Big(\phi\big(S_2(h'\ns{2}),S_1(h\ns{2})\big)\ot\psi\big(S_2(h'\ns{1}),S_1(h\ns{1})\big) \Big),
\end{split}
\end{align}
satisfies
\[
S_{\bowtie}((\d_1,h')(h,\d_2)) = (S_1(h),S_2(h'))
\]
for any $h\in \C{H}_1$, and any $h'\in \C{H}_2$.
\end{lemma}

\begin{proof}
For any $h\in \C{H}_1$, and any $h'\in \C{H}_2$ we have
\begin{align*}
& S_{\bowtie}((\d_1,h')(h,\d_2)) = S_{\bowtie}(\phi(h'\ns{1},h\ns{1}),\psi(h'\ns{2},h\ns{2})) = S_{\bowtie}(\phi(h'\ns{2},h\ns{2}),\psi(h'\ns{1},h\ns{1})) =\\
& \Big(\d_1,S_2\big(\psi(h'\ns{1},h\ns{1})\big)\Big)\Big( S_1\big(\phi(h'\ns{2},h\ns{2})\big),\d_2\Big) = \\
& \Big(\phi\big(S_2(\psi(h'\ns{1},h\ns{1}))\ns{1},S_1(\phi(h'\ns{2},h\ns{2}))\ns{1}\big), \psi\big(S_2(\psi(h'\ns{1},h\ns{1}))\ns{2},S_1(\phi(h'\ns{2},h\ns{2}))\ns{2}\big)\Big) = \\
& \Big(\phi\big(S_2(\psi(h'\ns{1},h\ns{1}))\ns{2},S_1(\phi(h'\ns{2},h\ns{2}))\ns{2}\big), \psi\big(S_2(\psi(h'\ns{1},h\ns{1}))\ns{1},S_1(\phi(h'\ns{2},h\ns{2}))\ns{1}\big)\Big) = \\
& \Big(\phi\big(S_2(\psi(h'\ns{1}\ns{1},h\ns{1}\ns{1})),S_1(\phi(h'\ns{2}\ns{1},h\ns{2}\ns{1}))\big), \psi\big(S_2(\psi(h'\ns{1}\ns{2},h\ns{1}\ns{2})),S_1(\phi(h'\ns{2}\ns{2},h\ns{2}\ns{2}))\big)\Big) =  \\
& \Big(\phi\big(S_2(\psi(h'\ns{1},h\ns{1})),S_1(\phi(h'\ns{3},h\ns{3}))\big), \psi\big(S_2(\psi(h'\ns{2},h\ns{2})),S_1(\phi(h'\ns{4},h\ns{4}))\big)\Big) =  \\
& \Big(\phi\big(S_2(\psi(h'\ns{1},h\ns{1})),S_1(\phi(h'\ns{2}\ns{2},h\ns{2}\ns{2}))\big), \psi\big(S_2(\psi(h'\ns{2}\ns{1},h\ns{2}\ns{1})),S_1(\phi(h'\ns{3},h\ns{3}))\big)\Big) = \\
& \Big(\phi\big(S_2(\psi(h'\ns{1},h\ns{1})),S_1(\phi(h'\ns{2}\ns{1},h\ns{2}\ns{1}))\big), \psi\big(S_2(\psi(h'\ns{2}\ns{2},h\ns{2}\ns{2})),S_1(\phi(h'\ns{3},h\ns{3}))\big)\Big) = \\
& \Big(\phi\big(S_2(\psi(h'\ns{1},h\ns{1})),S_1(\phi(h'\ns{2},h\ns{2}))\big), \psi\big(S_2(\psi(h'\ns{3},h\ns{3})),S_1(\phi(h'\ns{4},h\ns{4}))\big)\Big) = \\
& \Big(\phi\big(S_2(\psi(h'\ns{1},h\ns{1})),[\phi(\psi(h'\ns{2},h\ns{2}\ns{1}), S_1(h\ns{2}\ns{2}))]\big), \psi\big(S_2(\psi(h'\ns{3},h\ns{3})),S_1(\phi(h'\ns{4},h\ns{4}))\big)\Big) = \\
& \Big(S_1(h\ns{1}), \psi\big(S_2(\psi(h'\ns{1},h\ns{2})),S_1(\phi(h'\ns{2},h\ns{2}))\big)\Big) = \\
& \Big(S_1(h\ns{1}), \psi\big(\psi(S_2(h'\ns{1}\ns{1}),\phi(h'\ns{1}\ns{2},h\ns{2})),S_1(\phi(h'\ns{2},h\ns{2}))\big)\Big) = \Big(S_1(h), S_2(h')\Big),
\end{align*}
where on the second, fifth, and ninth equations we used \eqref{unit-action-QR-matched-V-Hopf-lemma}, on the sixth equation we used \eqref{module-coalg-I-lemma} and \eqref{module-coalg-II-lemma}, on the eleventh equation we used the fact that 
\[
S_1(\phi(h',h)) = \phi(\psi(h',h\ns{1}),S_1(h\ns{2})),
\]
which follows from \eqref{unit-action-QR-matched-III-Hopf-a-lemma}, and on the twelfth equation we used \eqref{unit-action-QR-matched-II-Hopf-a'-lemma}. Finally, on the thirteenth equation we used 
\[
S_2(\psi(h',h)) = \psi(S_2(h'\ns{1}),\phi(h'\ns{2},h)),
\]
which is a consequence of \eqref{unit-action-QR-matched-IV-Hopf-a-lemma}, and on the fourteenth we used \eqref{unit-action-QR-matched-II-Hopf-b'-lemma}.
\end{proof}

We are now ready for the main result.

\begin{proposition}\label{prop-matched-pair-Hopf-quasigrs}
Let $(\C{H}_1,\mu_1,\eta_1,\D_1,\ve_1,S_1)$ be an $m_1$-inverse Hopf quasigroups such that $S_1(\d_1) = \d_1$, and that $S_1^{h_1} \in {\rm Aut}(\C{H}_1)$, and $(\C{H}_2,\mu_2,\eta_2,\D_2,\ve_2,S_2)$ be an $m_2$-inverse Hopf quasigroup such that $S_2(\d_2) = \d_2$, and that $S_2^{h_2} \in {\rm Aut}(\C{H}_2)$. Furthermore, let there be two maps $\phi:\C{H}_2 \ot \C{H}_1 \to \C{H}_1$ and $\psi:\C{H}_2 \ot \C{H}_1 \to \C{H}_2$
satisfying
\begin{align}
& \phi(\d_2,h) = h, \quad \phi(h',\d_1) =\d_1, \qquad \psi(\d_2,h) = \d_2, \quad \psi(h',\d_1) =h', \label{unit-action-QR-matched-I-Hopf}\\
& \phi(S(h'\ns{1}),\phi(h'\ns{2},h)) = \ve_2(h')h = \phi(h'\ns{1},\phi(S(h'\ns{2}),h)) , \label{unit-action-QR-matched-II-Hopf-a'}\\
& \psi(\psi(h', S_1^m(h\ns{2}g)),S_1^{m+1}(h\ns{1})) = \psi(h',S_1^m(g)), \label{unit-action-QR-matched-II-Hopf-b}\\
& \psi(\psi(h',S(h\ns{1})),h\ns{2}) = \ve_1(h)h' = \psi(\psi(h',h\ns{1}),S(h\ns{2})) \label{unit-action-QR-matched-II-Hopf-b'}\\
& \D_1(\phi(h',h)) = \phi(h'\ns{1},h\ns{1}) \ot \phi(h'\ns{2},h\ns{2}), \qquad \ve_1(\phi(h',h)) = \ve_1(h)\ve_2(h'), \label{module-coalg-I}\\
& \D_2(\psi(h',h)) = \psi(h'\ns{1},h\ns{1}) \ot \psi(h'\ns{2},h\ns{2}), \qquad \ve_2(\psi(h',h)) = \ve_1(h)\ve_2(h'), \label{module-coalg-II}\\
& \phi(h',S_1^m(g)) = \begin{cases}
\phi(h'\ns{1},S_1^m(h\ns{3}g\ns{2}))\phi(\psi(h'\ns{2},S_1^m(h\ns{2}g\ns{1})),S_1^{m+1}(h\ns{1}))  & \text{if  } m = 2\ell+1, \\
\phi(h'\ns{1},S_1^m(h\ns{2}g\ns{1}))\phi(\psi(h'\ns{2},S_1^m(h\ns{3}g\ns{2})),S_1^{m+1}(h\ns{1})) & \text{if  } m = 2\ell,
\end{cases}  \label{unit-action-QR-matched-III-Hopf}\\
& \phi(h'\ns{1},S(h\ns{2}))\big[\phi(\psi(h'\ns{2},S(h\ns{1})),h\ns{3})\big] = \ve_1(h)\ve_2(h') =  \phi(h'\ns{1}, h\ns{1})\big[\phi(\psi(h'\ns{2},h\ns{2}),S(h\ns{3}))\big], \label{unit-action-QR-matched-III-Hopf-a}\\
& \big[\psi(S(h'\ns{1}),\phi(h'\ns{2},h\ns{1}))\big]\psi(h'\ns{3},h\ns{2}) = \ve_1(h)\ve_2(h') = \label{unit-action-QR-matched-IV-Hopf-a}\\
& \hspace{6cm} \big[\psi(h'\ns{1},\phi(S(h'\ns{3}),h\ns{1}))\big]\psi(S(h'\ns{2}),h\ns{2}), \notag\\
& \psi(h'\ns{1},h\ns{1}) \ot \phi(h'\ns{2},h\ns{2}) = \psi(h'\ns{2},h\ns{2}) \ot \phi(h'\ns{1},h\ns{1})  \label{unit-action-QR-matched-V-Hopf}
\end{align}
for any $h,g \in \C{H}_1$, any $h',g' \in \C{H}_2$, and any $m \in \B{Z}$ that satisfies 
\begin{align}\label{cong-eqn-II-matched-Hopf}
\begin{split}
& m \equiv m_1\,\, ({\rm mod}\,h_1), \\
& m \equiv m_2\,\, ({\rm mod}\,h_2).
\end{split}
\end{align}
Then  $\Big(\C{H}_1 \bowtie \C{H}_2 :=\C{H}_1 \ot \C{H}_2, \mu_{\bowtie},\eta_\ot,\D_\ot,\ve_\ot, S_{\bowtie}\Big)$ 
is an $m$-invertible Hopf quasigroup with the multiplication
\begin{equation}\label{matched-pair-multp-Hopf}
\mu_{\bowtie}((h\ot h')\ot(g\ot g'))=:(h\ot h')(g\ot g') := \Big(h\phi(h'\ns{1},g\ns{1}),\psi(h'\ns{2},g\ns{2})g'\Big),
\end{equation}
and the antipode
\begin{align}\label{perm-matched-pair-prod-Hopf}
\begin{split}
& S_{\bowtie}: \C{H}_1 \bowtie \C{H}_2 \to \C{H}_1 \bowtie \C{H}_2, \\ 
& S_{\bowtie}(h\ot h') := (\d_1\ot S_2(h'))(S_1(h)\ot \d_2) = \Big(\phi\big(S_2(h'\ns{2}),S_1(h\ns{2})\big)\ot\psi\big(S_2(h'\ns{1}),S_1(h\ns{1})\big) \Big),
\end{split}
\end{align}
if and only if
\begin{align}\label{m-inverse-cond-matched-Hopf}
& \begin{cases}
\left. \begin{array}{c}
\phi(h',h) = h, \\ 
\psi(h',h) = h', 
\end{array}\right\}
& \text{\rm if}\,\, m = 2\ell, \\
\left. \begin{array}{c}
\phi(S_2^m(\psi(h'\ns{2},g\ns{2})g'),S_1^m(\phi(h'\ns{1},g\ns{1}))) = \ve_2(h')\phi(S_2^m(g'),S_1^m(g)), \\
\psi(S_2^m(\psi(h'\ns{3},g\ns{2})g'),S_1^m(\phi(h'\ns{2},g\ns{1})))S_2^{m+1}(h'\ns{1}) = \ve_2(h')\psi(S_2^m(g'),S_1^m(g)),
\end{array}\right\}
& \text{\rm if}\,\, m = 2\ell+1,
\end{cases}
\end{align}
for any $h,g \in \C{H}_1$, and any $h',g'\in \C{H}_2$.
\end{proposition}

\begin{proof}
Let us first assume that the conditions \eqref{m-inverse-cond-matched-Hopf} are met. We shall begin with the observation that
\begin{align*}
& (h\ns{1},h'\ns{1})S_{\bowtie}(h\ns{2},h'\ns{2}) = \big[(h\ns{1},\d_2)(\d_1,h'\ns{1})\big]\big[(\d_1,S_2(h'\ns{2}))(S_1(h\ns{2}),\d_2)\big] = \\
& \big[(h\ns{1},\d_2)(\d_1,h'\ns{1})\big]\Big(\phi(S_2(h'\ns{3}),S_1(h\ns{3})),\psi(S_2(h'\ns{2}),S_1(h\ns{2}))\Big) = \\
& (h\ns{1},\d_2)\Big[(\d_1,h'\ns{1})\big(\phi(S_2(h'\ns{3}),S_1(h\ns{3})),\psi(S_2(h'\ns{2}),S_1(h\ns{2}))\big)\Big] = \\
& (h\ns{1},\d_2)\Big(\phi(h'\ns{1}\ns{1},\phi(S_2(h'\ns{3}),S_1(h\ns{3}))\ns{1}), \\
&\hspace{3cm} \psi(h'\ns{1}\ns{2},\phi(S_2(h'\ns{3}),S_1(h\ns{3}))\ns{2})\psi(S_2(h'\ns{2}),S_1(h\ns{2}))\Big) = \\
& (h\ns{1},\d_2)\Big(\phi(h'\ns{1},\phi(S_2(h'\ns{5}),S_1(h\ns{4}))), \psi(h'\ns{2},\phi(S_2(h'\ns{4}),S_1(h\ns{3})))\psi(S_2(h'\ns{3}),S_1(h\ns{2}))\Big) = \\
& (h\ns{1},\d_2)\Big(\phi(h'\ns{1},\phi(S_2(h'\ns{4}),S_1(h\ns{3}))), \psi(h'\ns{2}S_2(h'\ns{3}),S_1(h\ns{2}))\Big) = \\
& (h\ns{1},\d_2)\Big(\phi(h'\ns{1},\phi(S_2(h'\ns{2}),S_1(h\ns{2}))), \d_2\Big) = \\
& (h\ns{1},\d_2)(S_1(h\ns{2}), \ve_2(h')\d_2) = (h\ns{1}S_1(h\ns{2}), \d_2) = (\ve_1(h)\d_1,\ve_2(h')\d_2),
\end{align*}
where on the fifth equality we used \eqref{module-coalg-I} and \eqref{module-coalg-II}, on the sixth equality \eqref{unit-action-QR-matched-IV-Hopf-a}, and on the eighth equality we use \eqref{unit-action-QR-matched-II-Hopf-a'}. Similarly,
\begin{align*}
& S_{\bowtie}(h\ns{1},h'\ns{1})(h\ns{2},h'\ns{2}) = \big[(\d_1,S_2(h'\ns{1}))(S_1(h\ns{1}),\d_2)\big]\big[(h\ns{2},\d_2)(\d_1,h'\ns{2})\big] = \\
& \Big(\phi(S_2(h'\ns{2}),S_1(h\ns{2})),\psi(S_2(h'\ns{1}),S_1(h\ns{1}))\Big)\big[(h\ns{3},\d_2)(\d_1,h'\ns{3})\big] = \\
& \Big[\Big(\phi(S_2(h'\ns{2}),S_1(h\ns{2})),\psi(S_2(h'\ns{1}),S_1(h\ns{1}))\Big)(h\ns{3},\d_2)\Big](\d_1,h'\ns{3}) = \\
& \Big(\phi(S_2(h'\ns{2}),S_1(h\ns{2}))\phi(\psi(S_2(h'\ns{1}),S_1(h\ns{1}))\ns{1},h\ns{3}\ns{1}), \\
&\hspace{5cm}\psi(\psi(S_2(h'\ns{1}),S_1(h\ns{1}))\ns{2},h\ns{3}\ns{2})\Big)(\d_1,h'\ns{3}) = \\
& \Big(\phi(S_2(h'\ns{2}),S_1(h\ns{2}))\phi(\psi(S_2(h'\ns{1}\ns{2}),S_1(h\ns{1}\ns{2})),h\ns{3}\ns{1}), \\
&\hspace{5cm} \psi(\psi(S_2(h'\ns{1}\ns{1}),S_1(h\ns{1}\ns{1})),h\ns{3}\ns{2})\Big)(\d_1,h'\ns{3}) = \\
& \Big(\phi(S_2(h'\ns{3}),S_1(h\ns{3}))\phi(\psi(S_2(h'\ns{2}),S_1(h\ns{2})),h\ns{4}),\psi(\psi(S_2(h'\ns{1}),S_1(h\ns{1})),h\ns{5})\Big)(\d_1,h'\ns{4}) = \\
& \Big(\phi\big(S_2(h'\ns{2}),S_1(h\ns{2})h\ns{3}\big),\psi\big(\psi(S_2(h'\ns{1}),S_1(h\ns{1})),h\ns{4}\big)\Big)(\d_1,h'\ns{3}) =\\
& \Big(\d_1,\psi\big(\psi(S_2(h'\ns{1}),S_1(h\ns{1})),h\ns{2}\big)\Big)(\d_1,h'\ns{2}) = \Big(\ve_1(h)\d_1,S_2(h'\ns{1})\Big)(\d_1,h'\ns{2}) = \\
& (\ve_1(h)\d_1, S_2(h'\ns{1})h'\ns{2}) = (\ve_1(h)\d_1, \ve_2(h')\d_2),
\end{align*}
using \eqref{unit-action-QR-matched-III-Hopf-a} on the seventh equality, and \eqref{unit-action-QR-matched-II-Hopf-b'} on the tenth. Furthermore, \eqref{perm-matched-pair-prod-Hopf} is unique with the property \eqref{S-prop}. Indeed, if $T:\C{H}_1\ot \C{H}_2 \to \C{H}_1\ot \C{H}_2$, say $T(h,h') = (T_1(h,h'),T_2(h,h'))$, is a coalgebra anti-automorphism so that 
\begin{equation}\label{prop-T}
(h\ns{1},h'\ns{1})T(h\ns{2},h'\ns{2}) = (\ve_1(h)\d_1,\ve_2(h')\d_2) = T(h\ns{1},h'\ns{1})(h\ns{2},h'\ns{2}),
\end{equation}
then on one hand (from the first equality of \eqref{prop-T}) 
\begin{align}\label{S-prop-II}
\begin{split}
& (\ve_1(h)\d_1,\ve_2(h')\d_2) = (h\ns{1},h'\ns{1})T(h\ns{2},h'\ns{2}) = \\
& (h\ns{1},h'\ns{1})\Big(T_1(h\ns{2},h'\ns{2}), T_2(h\ns{2},h'\ns{2})\Big) = \\
& \Big(h\ns{1}\phi\big(h'\ns{1}\ns{1},T_1(h\ns{2},h'\ns{2})\ns{1}\big), \psi\big(h'\ns{1}\ns{2},T_1(h\ns{2},h'\ns{2})\ns{2}\big)T_2(h\ns{2},h'\ns{2})\Big),
\end{split}
\end{align}
while on the other hand (this time from the second equality of \eqref{prop-T}),
\begin{align}\label{S-prop-III}
\begin{split}
& (\ve_1(h)\d_1,\ve_2(h')\d_2) = T(h\ns{1},h'\ns{1})(h\ns{2},h'\ns{2}) = \\
& \Big(T_1(h\ns{1},h'\ns{1}), T_2(h\ns{1},h'\ns{1})\Big)(h\ns{2},h'\ns{2}) = \\
& \Big(T_1(h\ns{1},h'\ns{1})\phi\big(T_2(h\ns{1},h'\ns{1})\ns{1},h\ns{2}\ns{1}\big), \psi\big(T_2(h\ns{1},h'\ns{1})\ns{2},h\ns{2}\ns{2}\big)h'\ns{2}\Big). 
\end{split}
\end{align}
Application of $\Id\ot \ve_2:\C{H}_1\ot \C{H}_2\to \C{H}_1$ to \eqref{S-prop-II} yields
\[
h\ns{1}\phi\big(h'\ns{1},T_1(h\ns{2},h'\ns{2})\big) = \ve_1(h)\ve_2(h')\d_1,
\]
which, in turn, leads to
\[
\phi\big(h'\ns{1},T_1(h,h'\ns{2})\big) = \ve_2(h')S_1(h).
\]
But then,
\begin{equation}\label{T-I}
T_1(h,h') = \phi\Big(S_2(h'\ns{1}),\phi\big(h'\ns{2},T_1(h,h'\ns{3})\big)\Big) = \phi\big(S_2(h'),S_1(h)\big).
\end{equation}
Similarly, applying $\ve_1\ot \Id:\C{H}_1\ot \C{H}_2\to \C{H}_2$ to \eqref{S-prop-III} we derive
\begin{equation}\label{T-II}
T_2(h,h') = \psi\big(S_2(h'),S_1(h)\big).
\end{equation}
Now, from \eqref{T-I} and \eqref{T-II} we conclude $T = S_{\bowtie}$.

We next proceed to show that \eqref{perm-matched-pair-prod-Hopf} satisfies \eqref{S-m-prop}. In case of $m = 2\ell+1$, we have
\begin{align}\label{m=2l+1-matched-Hopf}
\begin{split}
& S_{\bowtie}^m\big((h\ns{2},h'\ns{2})(g,g')\big)S_{\bowtie}^{m+1}(h\ns{1},h'\ns{1}) =  S_{\bowtie}^m\big(h\ns{2}\phi(h'\ns{2},g\ns{1}),\psi(h'\ns{3},g\ns{2})g'\big)S_{\bowtie}^{m+1}(h\ns{1},h'\ns{1}) = \\
& \Big[\big(\d_1,S_2^m(\psi(h'\ns{3},g\ns{2})g')\big)\big(S_1^m(h\ns{2}\phi(h'\ns{2},g\ns{1})),\d_2\big)\Big]\big(S_1^{m+1}(h\ns{1}),S_2^{m+1}(h'\ns{1})\big) = \\
& \Big(\phi\big(S_2^m(\psi(h'\ns{3},g\ns{2})g')\ns{1}, S_1^m(h\ns{2}\phi(h'\ns{2},g\ns{1}))\ns{1}\big), \\
& \hspace{3cm}\psi\big(S_2^m(\psi(h'\ns{3},g\ns{2})g')\ns{2}, S_1^m(h\ns{2}\phi(h'\ns{2},g\ns{1}))\ns{2}\big)\Big) \big(S_1^{m+1}(h\ns{1}),S_2^{m+1}(h'\ns{1})\big) = \\
& \Big(\phi\big(S_2^m(\psi(h'\ns{3},g\ns{2})g')\ns{1}, S_1^m(h\ns{2}\phi(h'\ns{2},g\ns{1}))\ns{1}\big)\times\\
& \hspace{1.5cm}\phi\big(\psi\big(S_2^m(\psi(h'\ns{3},g\ns{2})g')\ns{2}\ns{1}, S_1^m(h\ns{3}\phi(h'\ns{2},g\ns{1}))\ns{2}\ns{1}, S_1^{m+1}(h\ns{1})\ns{1}\big)\big), \\
& \hspace{1cm}\psi\big(\psi\big(S_2^m(\psi(h'\ns{3},g\ns{2})g')\ns{2}\ns{2}, S_1^m(h\ns{2}\phi(h'\ns{2},g\ns{1}))\ns{2}\ns{2}, S_1^{m+1}(h\ns{1})\ns{2}\big)\big)S_2^{m+1}(h'\ns{1})\Big) = \\
& \Big(\phi\big(S_2^m(\psi(h'\ns{3},g\ns{2})g')\ns{1},\big[S_1^m(h\ns{2}\phi(h'\ns{2},g\ns{1}))\ns{1}S_1^{m+1}(h\ns{1})\ns{1}\big]\big), \\
&\hspace{2cm}\psi\big(S_2^m(\psi(h'\ns{3},g\ns{2})g')\ns{2},\big[S_1^m(h\ns{2}\phi(h'\ns{2},g\ns{1}))\ns{2}S_1^{m+1}(h\ns{1})\ns{2}\big]\big)S_2^{m+1}(h'\ns{1})\Big) = \\
& \Big(\ve_1(h)\phi\big(S_2^m(g')\ns{1},S_1^m(g)\ns{1}\big), \ve_2(h')\psi\big(S_2^m(g')\ns{2},S_1^m(g)\ns{2}\big)\Big) = (\d_1,\ve_2(h')S_2^m(g'))(\ve_1(h)S_1^m(g),\d_2) = \\
& \ve_1(h)\ve_2(h')S_{\bowtie}^m(g,g'),
\end{split}
\end{align} 
where on the second and the eighth equalities we used Lemma \ref{lemma-antipode}, on the fifth equality we used \eqref{unit-action-QR-matched-III-Hopf} and \eqref{unit-action-QR-matched-II-Hopf-b}, and on the sixth equality we used \eqref{m-inverse-cond-matched-Hopf}. If, on the other hand, $m=2\ell$
\begin{align}\label{m=2l-matched-Hopf}
\begin{split}
& S_{\bowtie}^m\big((h\ns{2},h'\ns{2})(g,g')\big)S_{\bowtie}^{m+1}(h\ns{1},h'\ns{1}) =  S_{\bowtie}^m\big(h\ns{2}\phi(h'\ns{2},g\ns{1}),\psi(h'\ns{3},g\ns{2})g'\big)S_{\bowtie}^{m+1}(h\ns{1},h'\ns{1}) = \\
& \Big(S_1^m(h\ns{2}\phi(h'\ns{2},g\ns{1})),S_2^m(\psi(h'\ns{3},g\ns{2})g')\Big)\Big[\big(\d_1,S_2^{m+1}(h'\ns{1})\big)\big(S_1^{m+1}(h\ns{1}),\d_2\big)\Big] = \\
& \Big(S_1^m(h\ns{2}\phi(h'\ns{2},g\ns{1})),S_2^m(\psi(h'\ns{3},g\ns{2})g')\Big) \times \\
&\Big(\phi\big(S_2^{m+1}(h'\ns{1})\ns{1},S_1^{m+1}(h\ns{1})\ns{1}\big), \psi\big(S_2^{m+1}(h'\ns{1})\ns{2},S_1^{m+1}(h\ns{1})\ns{2}\big)\Big) = \\
& \Big(S_1^m(h\ns{2}\phi(h'\ns{2},g\ns{1}))\Big[\phi\big(S_2^m(\psi(h'\ns{3},g\ns{2})g')\ns{1}, \phi\big(S_2^{m+1}(h'\ns{1})\ns{1}\ns{1},S_1^{m+1}(h\ns{1})\ns{1}\ns{1}\big)\big)\Big], \\
& \Big[\psi\big(S_2^m(\psi(h'\ns{3},g\ns{2})g')\ns{2}, \phi\big(S_2^{m+1}(h'\ns{1})\ns{1}\ns{2},S_1^{m+1}(h\ns{1})\ns{1}\ns{2}\big)\big)\Big] \times \\
&\hspace{7,5cm}\psi\big(S_2^{m+1}(h'\ns{1})\ns{2},S_1^{m+1}(h\ns{1})\ns{2}\big)\Big) = \\
& \big(\ve_1(h)S_1^m(g),\ve_2(h')S_2^m(g')\big),
\end{split}
\end{align}
where on the second equality we used Lemma \ref{lemma-antipode}, and on the fifth equality we used \eqref{m-inverse-cond-matched-Hopf}.

Conversely, let $\C{H}_1$ and $\C{H}_2$ be subject to the hypothesis of the theorem. Then, in the case of $m=2\ell+1$, the application of $\Id\ot \ve_2:\C{H}_1\ot \C{H}_2 \to \C{H}_1$ to the sixth equality
\begin{align*}
& \Big(\phi\big(S_2^m(\psi(h'\ns{3},g\ns{2})g')\ns{1},S_1^m(\phi(h'\ns{2},g\ns{1}))\ns{1}\big), \\
&\hspace{4cm}\psi\big(S_2^m(\psi(h'\ns{3},g\ns{2})g')\ns{2},S_1^m(\phi(h'\ns{2},g\ns{1}))\ns{2}\big)S_2^{m+1}(h'\ns{1})\Big) = \\
& \Big(\phi\big(S_2^m(g')\ns{1},S_1^m(g)\ns{1}\big), \ve_2(h')\psi\big(S_2^m(g')\ns{2},S_1^m(g)\ns{2}\big)\Big)
\end{align*}
of \eqref{m=2l+1-matched-Hopf} yields
\[
\phi\big(S_2^m(\psi(h'\ns{2},g\ns{2})g'),S_1^m(\phi(h'\ns{1},g\ns{1}))\big) = \phi\big(S_2^m(g'),S_1^m(g)\big)\ve_2(h')
\]
for any $g\in \C{H}_1$, and any $g',h'\in \C{H}_2$. Similarly, the application of $\ve_1\ot \Id:\C{H}_1\ot \C{H}_2\to\C{H}_2$ on the other hand (to the sixth equality of \eqref{m=2l+1-matched-Hopf}) this times yields
\[
\psi\big(S_2^m(\psi(h'\ns{3},g\ns{2})g'),S_1^m(\phi(h'\ns{2},g\ns{1}))\big)S_2^{m+1}(h'\ns{1}) = \ve_2(h')\psi\big(S_2^m(g'),S_1^m(g)\big).
\]
Next, if $m=2\ell$, then we apply $\Id\ot \ve_2:\C{H}_1\ot \C{H}_2 \to \C{H}_1$ to the fifth equality
\begin{align*}
& \Big(S_1^m(h\ns{2}\phi(h'\ns{2},g\ns{1}))\Big[\phi\big(S_2^m(\psi(h'\ns{3},g\ns{2})g')\ns{1}, \phi\big(S_2^{m+1}(h'\ns{1})\ns{1}\ns{1},S_1^{m+1}(h\ns{1})\ns{1}\ns{1}\big)\big)\Big], \\
& \Big[\psi\big(S_2^m(\psi(h'\ns{3},g\ns{2})g')\ns{2}, \phi\big(S_2^{m+1}(h'\ns{1})\ns{1}\ns{2},S_1^{m+1}(h\ns{1})\ns{1}\ns{2}\big)\big)\Big] \times\\
& \hspace{8cm}\psi\big(S_2^{m+1}(h'\ns{1})\ns{2},S_1^{m+1}(h\ns{1})\ns{2}\big)\Big) = \\
& \big(\ve_1(h)S_1^m(g),\ve_2(h')S_2^m(g')\big)
\end{align*}
of \eqref{m=2l-matched-Hopf} to get
\[
S_1^m(h\ns{2}\phi(h'\ns{2},g\ns{1}))\Big[\phi\big(S_2^m(\psi(h'\ns{3},g\ns{2})g'), \phi\big(S_2^{m+1}(h'\ns{1}),S_1^{m+1}(h\ns{1})\big)\big)\Big] = \ve_1(h)S_1^m(g)\ve_2(h')\ve_2(g')
\]
for any $g,h\in \C{H}_1$, and any $g',h'\in \C{H}_2$. In particular, for $h=1$ and $g'=1$ we arrive at
\[
S_1^m(\phi(h',g)) = \ve_2(h')S_1^m(g),
\] 
from which we conclude that
\begin{equation}\label{eqn-phi-triv}
\phi(h',g) = \ve_2(h')g.
\end{equation}
Similarly, the application of $\ve_1\ot \Id:\C{H}_1\ot \C{H}_2\to\C{H}_2$ to the fifth equality of \eqref{m=2l-matched-Hopf} yields
\begin{align*}
& \Big[\psi\big(S_2^m(\psi(h'\ns{3},g\ns{2})g'), \phi\big(S_2^{m+1}(h'\ns{1})\ns{1},S_1^{m+1}(h\ns{1})\ns{1}\big)\big)\Big]\psi\big(S_2^{m+1}(h'\ns{1})\ns{2},S_1^{m+1}(h\ns{1})\ns{2}\big)\Big) = \\
& \ve_1(h)\ve_1(g)\ve_2(h')S_2^m(g').
\end{align*}
Now, invoking \eqref{eqn-phi-triv}, and setting $g=1$ and $h'=1$, we obtain (in view of \eqref{unit-action-QR-matched-I-Hopf})
\[
\psi\big(S_2^m(g'), S_1^{m+1}(h)\big) = \ve_1(h)S_2^m(g'),
\]
from which the the triviality of the left action follows.
\end{proof}

\begin{definition}
Let $(\C{H}_1,\mu_1,\eta_1,\D_1,\ve_1,S_1)$ be an $m_1$-inverse Hopf quasigroup such that $S_1(\d_1) = \d_1$, and that $S_1^{h_1} \in {\rm Aut}(\C{H}_1)$, and $(\C{H}_2,\mu_2,\eta_2,\D_2,\ve_2,S_2)$ be an $m_2$-inverse Hopf quasigroup such that $S_2(\d_2) = \d_2$, and that $S_2^{h_2} \in {\rm Aut}(\C{H}_2)$. Let also $m\in \B{Z}$ be a solution of 
\begin{align*}
& m \equiv m_1\,\, ({\rm mod}\,h_1), \\
& m \equiv m_2\,\, ({\rm mod}\,h_2).
\end{align*} 
Then, $(\C{H}_1,\C{H}_2)$ is called a ``matched pair of $m$-inverse property Hopf quasigroups'' if the Hopf quasigroups $(\C{H}_1,\mu_1,\eta_1,\D_1,\ve_1,S_1)$ and $(\C{H}_2,\mu_2,\eta_2,\D_2,\ve_2,S_2)$ satisfy the conditions \eqref{unit-action-QR-matched-I-Hopf}-\eqref{unit-action-QR-matched-V-Hopf}.
\end{definition}

A remark is in order.

\begin{remark}
Given an $m_1$-inverse property quasigroup $Q_1$, an $m_2$-inverse property quasigroup $Q_2$, and a solution $m\in \B{Z}$ of 
\begin{align*}
& m \equiv m_1\,\, ({\rm mod}\,h_1), \\
& m \equiv m_2\,\, ({\rm mod}\,h_2),
\end{align*}  
let $\Big((Q_1,J_1,\d_1),(Q_2,J_2,\d_2)\Big)$ be a matched pair of $m$-inverse property quasigroups such that $J_1(q)q = \d_1$ for any $q\in Q_1$ and $J_2(q')q' = \d_2$ for any $q'\in Q_2$. Then 
$(kQ_1,kQ_2)$ is a matched pair of $m$-inverse property Hopf quasigroups.
\end{remark}

The following result is the universal property of the matched pair construction for $m$-inverse property Hopf quasigroups, that is, the analogue of  \cite[Thm. 7.2.3]{Majid-book}.

\begin{proposition}\label{prop-m-inv-loop-univ}
Let $(\C{H}_1,\mu_1,\eta_1,\D_1,\ve_1,S_1)$ be an $m_1$-inverse Hopf quasigroups such that $S_1(\d_1) = \d_1$, and that $S_1^{h_1} \in {\rm Aut}(\C{H}_1)$, and $(\C{H}_2,\mu_2,\eta_2,\D_2,\ve_2,S_2)$ be an $m_2$-inverse Hopf quasigroup such that $S_2(\d_2) = \d_2$, and that $S_2^{h_2} \in {\rm Aut}(\C{H}_2)$. Let also $m\in \B{Z}$ be a solution of 
\begin{align*}
& m \equiv m_1\,\, ({\rm mod}\,h_1), \\
& m \equiv m_2\,\, ({\rm mod}\,h_2),
\end{align*} 
and $\C{G}$ be an $m$-inverse Hopf quasigroup so that $\C{H}_1$ and $\C{H}_2$ are $m$-inverse Hopf quasi-subgroups of $\C{G}$;
\[
\C{H}_1 \xhookrightarrow{} \C{G} \xhookleftarrow{} \C{H}_2,
\]
such that the multiplication on $\C{G}$ yields an isomorphism
\begin{equation}\label{Theta-map-Hopf}
\Theta:\C{H}_1\otimes \C{H}_2 \to \C{G}, \qquad h\ot h' \mapsto hh',
\end{equation}
of vector spaces, under which the multiplications are compatible as
\[
(hh')g = h(h'g), \qquad g(hh') = (gh)h',
\]
for any $h\in \C{H}_1$, any $h'\in \C{H}_2$, and any $g\in \C{G}$, while the antipodes are compatible as 
\begin{equation}\label{J-on-Q-Hopf}
S(hh') = S_2(h')S_1(h), \qquad S(h'h) = S_1(h)S_2(h')
\end{equation}
for any $h\in \C{H}_1$, any $h' \in \C{H}_2$, and any $g\in \C{G}$. Then, $(\C{H}_1,\C{H}_2)$ is a matched pair of $m$-inverse Hopf quasigroups, and $\C{G} \cong \C{H}_1\bowtie \C{H}_2$ as Hopf quasigroups.
\end{proposition}

\begin{proof}
Let us begin with the mappings
\begin{equation}\label{phi-and-psi-Hopf}
\phi:\C{H}_2\otimes \C{H}_1 \to \C{H}_1, \qquad \psi:\C{H}_2\otimes \C{H}_1 \to \C{H}_2
\end{equation}
given by
\[
\phi(h',h) := ((\Id\ot \ve_2)\circ \Theta^{-1})(h'h), \qquad \psi(h',h) := ((\ve_1\ot \Id)\circ \Theta^{-1})(h'h),
\]
through
\begin{equation}\label{multp-in-Q-and-RS-Hopf}
h'h = \Theta\big(\phi(h'\ns{1},h\ns{1}),\psi(h'\ns{2},h\ns{2})\big).
\end{equation}
It then follows at once that the isomorphism \eqref{Theta-map-Hopf} respect the multiplications in $\C{G}$ and $\C{H}_1\bowtie \C{H}_2$.

It remains to show that the mappings \eqref{phi-and-psi-Hopf} have the properties \eqref{unit-action-QR-matched-I-Hopf}-\eqref{unit-action-QR-matched-V-Hopf}.

The first one, \eqref{unit-action-QR-matched-I-Hopf}, follows from the consideration of $h=\d_1$ and $h'=\d_2$ in \eqref{multp-in-Q-and-RS-Hopf}, respectively.

Next, the linear map $\Psi: \C{H}_2 \ot \C{H}_1 \to \C{H}_1 \ot \C{H}_2$ given by 
\[
\Psi(h' \ot h):= \Theta^{-1}(h'h) = \phi(h'\ns{1},h\ns{1}) \ot \psi(h'\ns{2},h\ns{2})
\]
being a coalgebra homomorphism, we have
\[
\D_\ot\circ \Psi = (\Psi\ot \Psi)\circ \D_\ot, \qquad \big((\ve_1\ot \ve_2)\circ \Psi\big)(h\ot h') =\ve_1(h)\ve_2(h'), 
\]
for any $h\in \C{H}_1$, and any $h'\in\C{H}_2$. Applying on an arbitrary $h'\ot h \in \C{H}_2\ot \C{H}_1$, we arrive at
\begin{align*}
& \Big[\phi(h'\ns{1},h\ns{1})\ns{1} \ot \psi(h'\ns{2},h\ns{2})\ns{1}\Big] \ot \Big[\phi(h'\ns{1},h\ns{1})\ns{2} \ot \psi(h'\ns{2},h\ns{2})\ns{2}\Big] = \\
& \Big(\phi(h'\ns{1}\ns{1},h\ns{1}\ns{1}) \ot \psi(h'\ns{1}\ns{2},h\ns{1}\ns{2})\Big) \ot \Big(\phi(h'\ns{2}\ns{1},h\ns{2}\ns{1}) \ot \psi(h'\ns{2}\ns{2},h\ns{2}\ns{2})\Big).
\end{align*}
Now, $\Id \ot \ve_2 \ot \Id \ot \ve_2$ yields \eqref{module-coalg-I}, and $\ve_1 \ot\Id \ot \ve_1 \ot \Id$ results in \eqref{module-coalg-II}. Furthermore, $\ve_1 \ot \Id\ot \Id \ot \ve_2$ leads to \eqref{unit-action-QR-matched-V-Hopf}.

On the other hand, in view of \eqref{J-on-Q-Hopf} the property $g\ns{1}S(g\ns{2}) = \ve(g)\d$ implies $(h\ns{1}h'\ns{1})S(h\ns{2}h'\ns{2}) = \ve_1(h)\ve_2(h')\d$ for any $h\in \C{H}_1$ and any $h'\in \C{H}_2$, which in turn implies 
\begin{align*}
& \Big(h\ns{1}\phi(h'\ns{1},\phi(S_2(h'\ns{5}),S_1(h\ns{4}))), \psi(h'\ns{2},\phi(S_2(h'\ns{4}),S_1(h\ns{3})))\psi(S_2(h'\ns{3}),S_1(h\ns{2}))\Big) = \\
&\hspace{2cm}(h\ns{1}S_1(h\ns{2}), \d_2) = (\ve_1(h)\d_1,\ve_2(h')\d_2).
\end{align*}
We then obtain the second equality of \eqref{unit-action-QR-matched-II-Hopf-a'} by applying $\Id \ot \ve_2$, as well as the second equality of \eqref{unit-action-QR-matched-IV-Hopf-a} via $\ve_1\ot \Id$. Similarly, $S(g\ns{1})g\ns{2} = \ve(g)\d$ yields
\begin{align*}
& \Big(\phi(S_2(h'\ns{3}),S_1(h\ns{3}))\phi(\psi(S_2(h'\ns{2}),S_1(h\ns{2})),h\ns{4}),\phi(\psi(S_2(h'\ns{1}),S_1(h\ns{1})),h\ns{5})h'\ns{4}\Big) = \\
& (\ve_1(h)\d_1, S_2(h'\ns{1})h'\ns{2}) = (\ve_1(h)\d_1, \ve_2(h')\d_2),
\end{align*}
which in turn implies the first equality of \eqref{unit-action-QR-matched-III-Hopf-a} by $\Id \ot \ve_2$, and the first equality of \eqref{unit-action-QR-matched-II-Hopf-b'} by $\ve_1 \ot\Id$.

On the next step, $J^m_Q(qq')J^{m+1}_Q(q) = J^m_Q(q')$ for any $q,q'\in Q$ provides, along the lines of \eqref{m=2l+1-matched-Hopf},
\begin{align*}
& \Big(\phi\big(S_2^m(\psi(h'\ns{3},g\ns{2})g')\ns{1}, S_1^m(h\ns{2}\phi(h'\ns{2},g\ns{1}))\ns{1}\big)\times\\
& \hspace{1.5cm}\phi\big(\psi\big(S_2^m(\psi(h'\ns{3},g\ns{2})g')\ns{2}\ns{1}, S_1^m(h\ns{3}\phi(h'\ns{2},g\ns{1}))\ns{2}\ns{1}, S_1^{m+1}(h\ns{1})\ns{1}\big)\big), \\
& \hspace{1.5cm}\psi\big(\psi\big(S_2^m(\psi(h'\ns{3},g\ns{2})g')\ns{2}\ns{2}, S_1^m(h\ns{2}\phi(h'\ns{2},g\ns{1}))\ns{2}\ns{2}, S_1^{m+1}(h\ns{1})\ns{2}\big)\big)S_2^{m+1}(h'\ns{1})\Big) = \\
& \Big(\ve_1(h)\phi\big(S_2^m(g')\ns{1},S_1^m(g)\ns{1}\big), \ve_2(h')\psi\big(S_2^m(g')\ns{2},S_1^m(g)\ns{2}\big)\Big) = (\d_1,\ve_2(h')S_2^m(g'))(\ve_1(h)S_1^m(g),\d_2) = \\
& \ve_1(h)\ve_2(h')S_{\bowtie}^m(g,g').
\end{align*}
In particular, for $h'=\d_2$ we see that
\begin{align*}
& \Big(\phi\big(S_2^m(g')\ns{1}, S_1^m(h\ns{2}g\ns{1})\ns{1}\big)\phi\big(\psi\big(S_2^m(g')\ns{2}\ns{1}, S_1^m(h\ns{3}g\ns{1})\ns{2}\ns{1}, S_1^{m+1}(h\ns{1})\ns{1}\big)\big), \\
& \hspace{5.5cm}\psi\big(\psi\big(S_2^m(g')\ns{2}\ns{2}, S_1^m(h\ns{2}g\ns{1})\ns{2}\ns{2}, S_1^{m+1}(h\ns{1})\ns{2}\big)\big)\Big) = \\
& \Big(\ve_1(h)\phi\big(S_2^m(g')\ns{1},S_1^m(g)\ns{1}\big), \psi\big(S_2^m(g')\ns{2},S_1^m(g)\ns{2}\big)\Big),
\end{align*}
which implies \eqref{unit-action-QR-matched-III-Hopf} by $\Id \ot \ve_2$, and \eqref{unit-action-QR-matched-II-Hopf-b} by $\ve_1 \ot \Id$. Let us also remark that \eqref{unit-action-QR-matched-III-Hopf}  implies the second equality of \eqref{unit-action-QR-matched-III-Hopf-a}, and that \eqref{unit-action-QR-matched-II-Hopf-b} implies the second equation of \eqref{unit-action-QR-matched-II-Hopf-b'}.

Equipped with these now, \eqref{J-on-Q-Hopf} gives
\begin{align*}
& S_{\bowtie}((\d_1,h')(h,\d_2)) = \\
& \Big(\phi\big(S_2(\psi(h'\ns{1},h\ns{1})),S_1(\phi(h'\ns{2},h\ns{2}))\big), \psi\big(S_2(\psi(h'\ns{3},h\ns{3})),S_1(\phi(h'\ns{4},h\ns{4}))\big)\Big) = \\
& \Big(\phi\big(S_2(\psi(h'\ns{1},h\ns{1})),[\phi(\psi(h'\ns{2},h\ns{2}\ns{1}), S_1(h\ns{2}\ns{2}))]\big), \psi\big(S_2(\psi(h'\ns{3},h\ns{3})),S_1(\phi(h'\ns{4},h\ns{4}))\big)\Big) = \\
&\Big(S_1(h), S_2(h')\Big).
\end{align*}
Then, the application of $\Id\ot \ve_2$ yields
\[
\phi\big(S_2(\psi(h'\ns{1},h\ns{1})),[\phi(\psi(h'\ns{2},h\ns{2}), S_1(h\ns{3}))]\big) = \ve_2(h')S_1(h),
\]
in particular,
\begin{align*}
& \phi\big(S_2(\psi(\psi(h',S_1(h\ns{1}))\ns{1},h\ns{2}\ns{1})),[\phi(\psi(\psi(h',S_1(h\ns{1}))\ns{2},h\ns{2}\ns{2}), S_1(h\ns{2}\ns{3}))]\big) = \\
&\ve_2(\psi(h',S_1(h\ns{1})))S_1(h\ns{2}),
\end{align*}
that is, 
\[
\phi\big(S_2(h'\ns{1}),\phi(h'\ns{2}, S_1(h))\big) =\ve_2(h')S_1(h),
\]
the first equality of \eqref{unit-action-QR-matched-II-Hopf-a'}. Similarly, the application of $\ve_1\ot\Id$ onto 
\begin{align*}
& S_{\bowtie}((\d_1,h')(h,\d_2)) = \Big(S_1(h\ns{1}), \psi\big(S_2(\psi(h'\ns{1},h\ns{2})),S_1(\phi(h'\ns{2},h\ns{2}))\big)\Big) = \Big(S_1(h), S_2(h')\Big),
\end{align*}
implies
\[
\psi\big(S_2(\psi(h'\ns{1},h\ns{2})),S_1(\phi(h'\ns{2},h\ns{2}))\big)\Big) = \ve_1(h)S_2(h').
\]
Hence, we see that
\[
\psi\Big(\psi\big(S_2(\psi(h'\ns{1}\ns{1},h\ns{1}\ns{2})),S_1(\phi(h'\ns{1}\ns{2},h\ns{1}\ns{2}))\big),\phi(h'\ns{2},h\ns{2})\Big) = \ve_1(h\ns{1})\psi\big(S_2(h'\ns{1}),\phi(h'\ns{2},h\ns{2})\big),
\]
that is,
\[
S_2(\psi(h',h))= \psi\big(S_2(h'\ns{1}),\phi(h'\ns{2},h)\big).
\]
But then,
\[
\big[\psi(S(h'\ns{1}),\phi(h'\ns{2},h\ns{1}))\big]\psi(h'\ns{3},h\ns{2}) = S_2(\psi(h'\ns{1},h\ns{1}))\psi(h'\ns{2},h\ns{2}) =\ve_1(h)\ve_2(h') ,
\]
the first equality of \eqref{unit-action-QR-matched-IV-Hopf-a} is satisfied.

Finally, having obtained \eqref{unit-action-QR-matched-I-Hopf}-\eqref{unit-action-QR-matched-V-Hopf}, it is possible to derive \eqref{m-inverse-cond-matched-Hopf} from \eqref{m=2l+1-matched-Hopf} in the case $m=2\ell+1$, and from \eqref{m=2l-matched-Hopf} in the case $m=2\ell$.
\end{proof}

\bibliographystyle{plain}
\bibliography{references}{}

\end{document}